\documentclass[12 pt]{amsart}
\pdfoutput=1

\usepackage{amsmath,amsthm,amsfonts,amssymb,amscd,euscript,enumerate}
\usepackage{tikz}
\usetikzlibrary { decorations.pathmorphing, decorations.pathreplacing, decorations.shapes, } 
\usepackage[all,arc]{xy}
\usepackage[foot]{amsaddr}
\usepackage{enumerate}
\usepackage{mathrsfs} %for mathscr
\usepackage{thmtools}
\usepackage{datetime}
\usepackage[colorlinks, citecolor = darkspringgreen, linkcolor = melon]{hyperref}
\usepackage{graphicx}
\usepackage{color}
\usepackage{cleveref}
\usepackage{comment}
\usepackage{float}
\usepackage{bbm}
\input{xy}
\xyoption{all}
\usetikzlibrary{shapes,arrows.meta}
\usepackage{marginnote}
\usepackage{colonequals}
\definecolor{melon}{rgb}{0.4, 0.2, 1}
\definecolor{darkspringgreen}{rgb}{0.14, 0.7, 0.3}
\usepackage[symbol]{footmisc}

\definecolor{amethyst}{rgb}{0.6, 0.4, 0.8}
\definecolor{darkspringgreen}{rgb}{0.09, 0.45, 0.27}

\calclayout

\setboolean{@twoside}{false}

\newcommand\blfootnote[1]{%
  \begingroup
  \renewcommand\thefootnote{}\footnote{#1}%
  \addtocounter{footnote}{-1}%
  \endgroup
}

\newtheorem{thm}{Theorem}[section]
\newtheorem{cor}[thm]{Corollary}
\newtheorem{prop}[thm]{Proposition}
\newtheorem{lem}[thm]{Lemma}

\newtheorem{prob}[thm]{Problem}

\theoremstyle{definition}
\newtheorem{defn}[thm]{Definition}

\newtheorem{exmp}[thm]{Example}

\theoremstyle{remark}
\newtheorem{rem}[thm]{Remark}

\makeatletter
\let\c@equation\c@thm
\makeatother
\numberwithin{equation}{section}

\makeatletter
\newcommand*\bigcdot{\mathpalette\bigcdot@{.5}}
\newcommand*\bigcdot@[2]{\mathbin{\vcenter{\hbox{\scalebox{#2}{$\m@th#1\bullet$}}}}}
\makeatother

\makeatletter
\def\subsection{\@startsection{subsection}{3}%
  \z@{.5\linespacing\@plus.7\linespacing}{.1\linespacing}%
  {\bfseries}}
\makeatother

\newcommand{\Z}{\mathbb{Z}}

\newcommand{\R}{\mathbb{R}}

\newcommand{\calA}{\mathcal{A}}

\newcommand{\calC}{\mathcal{C}}

\newcommand{\calL}{\mathcal{L}}
\newcommand{\calP}{\mathcal{P}}

\newcommand{\calT}{\mathcal{T}}

\newcommand{\frakd}{\mathfrak{d}}

\newcommand{\frakm}{\mathfrak{m}}

\newcommand{\frakD}{\mathfrak{D}}

\renewcommand{\phi}{\varphi}
\newcommand{\Th}{\text{th}}

\DeclareMathOperator{\sh}{sh}
\DeclareMathOperator{\cas}{cas}
\DeclareMathOperator{\Supp}{Supp}

\DeclareMathOperator{\Mono}{Mono}

\DeclareMathOperator{\cut} { \setminus}

\DeclareMathOperator{\ttheta}{\widetilde \theta}

\DeclareMathOperator{\CP}{CP}
\DeclareMathOperator{\BL}{BL}
\DeclareMathOperator{\pro}{pro}
\newcommand{\init}{\text{in}}
\newcommand{\resp}{resp. }

\usepackage{etoolbox}
\makeatletter
\patchcmd{\@maketitle}
  {\ifx\@empty\@dedicatory}
  {\ifx\@empty\@date \else {\vskip3ex \centering\footnotesize\@date\par\vskip1ex}\fi
   \ifx\@empty\@dedicatory}
  {}{}
\patchcmd{\@adminfootnotes}
  {\ifx\@empty\@date\else \@footnotetext{\@setdate}\fi}
  {}{}{}
\makeatother

\makeatletter \def\l@subsection{\@tocline{2}{0pt}{1pc}{5pc}{}} \def\l@subsection{\@tocline{2}{0pt}{2pc}{6pc}{}} \makeatother

\renewcommand\vec[1]{\overrightarrow{#1}}

\def\AB#1{(\textcolor{blue}{AB: #1})}

\title[broken lines and compatible pairs for rank 2 quantum cluster algebras]{Broken lines and compatible pairs for rank 2 quantum cluster algebras}

\author[Burcroff]{Amanda Burcroff}
\address{\hspace{-0.4cm}Department of Mathematics, Harvard University, Cambridge, MA, USA 02138}
\email{\href{mailto:aburcroff@math.harvard.edu}{aburcroff@math.harvard.edu}}

\author[Lee]{Kyungyong Lee}
\address{Department of Mathematics, University of Alabama, Tuscaloosa, AL 35487, USA and Korea
Institute for Advanced Study, Seoul 02455, Republic of Korea}
\email{\href{mailto:kyungyong.lee@ua.edu}{kyungyong.lee@ua.edu} \& \href{mailto:klee1@kias.re.kr}{klee1@kias.re.kr}}

\begin{document}

\begin{abstract}
   There have been several combinatorial constructions of universally positive bases in cluster algebras, and these same combinatorial objects play a crucial role in the known proofs of the famous positivity conjecture for cluster algebras.  The greedy basis was constructed in rank $2$ by Lee--Li--Zelevinsky using compatible pairs on Dyck paths.  The theta basis, introduced by Gross--Hacking--Keel--Kontsevich, has elements expressed as a sum over broken lines on scattering diagrams.  It was shown by Cheung--Gross--Muller--Musiker--Rupel--Stella--Williams that these bases coincide in rank $2$ via algebraic methods, and they posed the open problem of giving a combinatorial proof by constructing a (weighted) bijection between compatible pairs and broken lines. 
   
   We construct a quantum-weighted bijection between compatible pairs and broken lines for the quantum type $A_2$ and the quantum Kronecker cluster algebras.  By specializing the quantum parameter, this handles the problem of Cheung \emph{et al.}\;for skew-symmetric cluster algebras of finite and affine type.  For cluster monomials in skew-symmetric rank-$2$ cluster algebras, we construct a quantum-weighted bijection between positive compatible pairs (which comprise almost all compatible pairs) and broken lines of negative angular momentum.
\end{abstract}

\maketitle
\setcounter{tocdepth}{1}
\tableofcontents

\section{Introduction} \label{sec: intro}
\emph{Cluster algebras}, initially introduced by Fomin and Zelevinsky \cite{FZ} as an algebraic framework for investigating dual canonical bases in semisimple groups, have evolved into a rich field with applications spanning combinatorics, algebraic geometry, and representation theory.   A cluster algebra of rank $n$ is constructed combinatorially from certain elements called \emph{cluster variables}, each of which can be expressed as a Laurent polynomial in $n$ \emph{initial cluster variables}.  One amazing property of these Laurent polynomials is that they have {\bf positive} integer coefficients.  This positivity property was conjectured by Fomin and Zelevinsky in 2002, and this conjecture remained open for over 10 years before being resolved by Lee--Schiffler \cite{LSpositive} (for cluster algebras from quivers) and Gross--Hacking--Keel--Kontsevich \cite{GHKK} (for cluster algebras of geometric type), as well as Davison \cite{davison2018positive} and Davison--Mandel \cite{davison2021strong} (for quantum cluster algebras from quivers).  The first two proofs rely upon expressing rank two cluster variables as a sum over combinatorial objects.  Our work focuses on establishing combinatorial connections between the two classes of objects that appear in these proofs of positivity: \emph{compatible pairs} corresponding to \emph{greedy basis} elements and \emph{broken lines} corresponding to \emph{theta basis} elements.  Though we work in the more general setting of \emph{quantum cluster algebras}, our results are new even in the classical setting. 

Lee, Li, and Zelevinsky \cite{LLZ} defined the \emph{greedy basis} for rank-$2$ cluster algebras, a basis consisting of indecomposable positive elements including the cluster monomials.  They provided a combinatorial formula \cite[Theorem 11]{LLZ} for the Laurent expansion of each greedy basis element as a sum over \emph{compatible pairs} (see, for example, \autoref{fig: cascade example}), which are pairs $(S_1,S_2)$ of edge sets in a maximal Dyck path where the set $S_1$ of horizontal edges and the set $S_2$ of vertical edges satisfy a compatibility condition. This expansion formula was later used in Lee and Schiffler's proof of positivity for cluster algebras from quivers \cite{LS3positive,LSpositive}. In the case of cluster variables, compatible pairs are in correspondence with certain colored subpaths of Dyck paths \cite{LS, Lin, Bur}. Rupel subsequently provided a non-commutative analogue of this expansion formula specifically for the cluster variable case, and this formula specializes to the quantum rank-$2$ cluster algebra setting \cite[Corollary 5.4]{Rup2}.  
 
 Gross, Hacking, Keel, and Kontsevich \cite{GHKK} proved the positivity property for cluster algebras of geometric type by establishing a novel connection between cluster algebras and \emph{scattering diagrams}, which arose earlier in the study of mirror symmetry \cite{GS, KS}.  They gave another expansion formula for cluster variables (in cluster algebras of arbitrary rank) as a sum over weights of piecewise linear curves called \emph{broken lines} on cluster scattering diagrams (see, for example, \autoref{fig: broken line q weight}).  This approach allowed them to construct the \emph{theta bases} for cluster algebras\footnote{More precisely, the theta basis is a basis for an algebra between the ordinary and upper cluster algebra. In the case of rank 2, the ordinary and upper cluster algebra coincide, so the theta basis is actually a basis for the cluster algebra.}.  Though the constructions appear rather different, for rank-$2$ cluster algebras the theta basis is the same as the greedy basis.
 \begin{thm}[Cheung, Gross, Muller, Musiker, Rupel, Stella, and Williams {\cite[Theorem 1.1]{CGM}}]\label{thm: greedy equals theta}
     The rank-$2$ greedy basis and theta basis coincide.
 \end{thm}
 
The proof of \autoref{thm: greedy equals theta} given in \cite{CGM} is algebraic rather than combinatorial.  In particular, the authors show that any \emph{pointed element} of a rank-$2$ cluster algebra with the same support as a greedy basis element must be a scalar multiple of it.  Since the theta basis elements are pointed, the authors proved restrictions on the broken line behavior that implied their support must match that of a greedy element.  This approach only handles the support of the theta basis elements rather than treating broken lines individually, and hence does not yield a combinatorial approach for understanding the connection between broken lines and compatible pairs. Cheung, Gross, Muller, Musiker, Rupel, Stella, and Williams thus posed the open problem of finding a combinatorial explanation for this phenomenon. 
 
 \begin{prob}[{\cite[Remark 5.6]{CGM}}]\label{prob: comb proof} Find a combinatorial proof of \autoref{thm: greedy equals theta} by constructing an explicit (weighted) bijection between broken lines and compatible pairs. \end{prob}

 \emph{Quantum cluster algebras} were introduced by Berenstein and Zelevinsky \cite{BZ} as non-commutative deformations of cluster algebras and are related to canonical bases in quantum groups.  We work inside the \emph{quantum torus} ${\mathcal{T} := \Z[q^{\pm 1}]\langle X_1^{\pm 1}, X_2^{\pm 1} : X_1X_2 = q^2X_2X_1\rangle}$. The \emph{quantum rank-$2$ $r$-Kronecker cluster algebra} $\calA_q(r,r)$ is the $\Z[q^{\pm 1}]$-subalgebra of the skew field of fractions of $\mathcal{T}$ generated by the \emph{quantum cluster variables} $\{X_n\}_{n \in \Z}$, which follow the recursion $X_{n+1}X_{n-1} = q^{r}X_n^r + 1$. The (classical) $r$-Kronecker cluster algebra $\calA(r,r)$ is obtained from $\calA_q(r,r)$ by specializing the parameter $q$ to $1$, and the $r$-Kronecker cluster algebras comprise all skew-symmetric rank-$2$ cluster algebras.  For more background on cluster algebras, see \cite{FWZ}.

 The greedy basis for rank-$2$ cluster algebras was extended to the quantum rank-$2$ setting by Lee, Li, Rupel, and Zelevinsky \cite{LLRZ}, though a quantum weighting on compatible pairs has only been constructed in the cluster variable case \cite{Rup2}.  The quantum theta basis was recently constructed by Davison and Mandel \cite{davison2021strong}, and they show that this basis satisfies the \emph{strong positivity property}.  Both the quantum greedy and quantum theta bases contain the \emph{quantum cluster monomials}, i.e., elements of the form $q^{\alpha\beta}X_n^\alpha X_{n+1}^\beta$ in $\calA_q(r,r)$ for nonnegative integers $\alpha,\beta$.  It is suggested by Davison and Mandel, though not yet proved, that the quantum rank-$2$ greedy and theta bases coincide.  

Our main results connect almost all objects involved in the Lee--Li--Zelevinsky and Gross--Hacking--Keel--Kontsevich formulas for the cluster monomials in the quantum $r$-Kronecker cluster algebra. We say that a map from compatible pairs to broken lines is a \emph{$q$-weighted bijection} if the sum of quantum weights of compatible pairs in the inverse image of each broken line equals the quantum weight of the broken line. 

In the case of the quantum \emph{type $A_2$} cluster algebra  $\calA_q(1,1)$ and the quantum \emph{Kronecker} cluster algebra $\calA_q(2,2)$, we can construct a $q$-weighted bijection between all compatible pairs corresponding to greedy basis elements and all broken lines corresponding to theta basis elements.  The greedy and theta bases for $\calA_q(1,1)$ consist entirely of quantum cluster monomials, so the two bases coincide.  However, there are elements of the quantum greedy and theta bases of $\calA_q(2,2)$ that are not quantum cluster monomials.  Building off prior work of \cite{mandel2023bracelets} and \cite{DX}, we show that the quantum greedy basis and quantum theta basis coincide for $\calA_q(2,2)$ (see \autoref{subsec: bases}).

We additionally describe a quantum weighting on the compatible pairs corresponding to greedy basis elements of $\calA_q(2,2)$ that are not cluster monomials.  Along with Rupel's quantum weighting for compatible pairs, this gives a quantum weighting on all compatible pairs corresponding to greedy basis elements of $\calA_q(2,2)$.  We then introduce the \emph{cascade} of a compatible pair and use this to construct a map to the broken lines for $\calA_q(2,2)$ that respects the quantum weights.

\begin{thm}[see \autoref{thm: phi q weighted bijection}, \autoref{thm: phi q weighted bijection kron pos}, and \autoref{thm: phi kron theta bijection}]
For $\calA_q(r,r)$ where $r = 1$ or $2$, there is an explicit $q$-weighted bijection between compatible pairs corresponding to quantum greedy basis elements and broken lines corresponding to quantum theta basis elements.  
\end{thm}

By specializing the parameter $q$ to $1$, this yields a solution to \autoref{prob: comb proof} for the type $A_2$ cluster algebra and the Kronecker cluster algebra. 

\begin{cor}
For the type $A_2$ cluster algebra $\calA(1,1)$ and the Kronecker cluster algebra $\calA(2,2)$, there is an explicit (weighted) bijection between compatible pairs and broken lines corresponding to theta basis elements.  
\end{cor}

While \autoref{prob: comb proof} seems quite difficult for $\calA(r,r)$ where $r > 2$, we can construct such a $q$-weighted bijection for a subclass of compatible pairs that arose in Lee and Schiffler's proof of the positivity property \cite[Theorem 3.22]{LSpositive}.  The \emph{positive compatible pairs} are the compatible pairs $(S_1,S_2)$ where $r|S_2|$ does not exceed the horizontal length of the Dyck path.\footnote{Without loss of generality, we assume that the horizontal length of the Dyck path is no less than the vertical length.}  Asymptotically, almost all compatible pairs corresponding to (quantum) cluster monomials are positive (see \autoref{lem: almost all}).  The positive compatible pairs correspond to broken lines with \emph{negative angular momentum} (see \autoref{subsec: ang momentum} for details).  For the (quantum) cluster monomials, these are precisely the broken lines that do not cross over the ``Badlands'' region of the scattering diagram.

\begin{thm}[see \autoref{thm: phi q weighted bijection}]
For quantum cluster monomials in $\calA_q(r,r)$ where $r > 2$, there is an explicit $q$-weighted bijection between positive compatible pairs (which comprise almost all compatible pairs) and broken lines of negative angular momentum.  
\end{thm}

By specializing $q$ to $1$, this yields a partial answer to \autoref{prob: comb proof} for skew-symmetric rank-$2$ cluster algebras.

\begin{cor}
For cluster monomials in $\calA(r,r)$ where $r > 2$, there is an explicit (weighted) bijection between positive compatible pairs (which comprise almost all compatible pairs) and broken lines of negative angular momentum.  
\end{cor}

A major obstruction in extending this bijection to the entire theta basis for $r > 2$ is that the corresponding scattering diagram is not well-understood.  In particular, there are infinitely many non-cluster walls that are dense in the full-dimensional ``Badlands'' region of the scattering diagram (see, for example, \cite[Figure 2]{Rea}) that have not been explicitly described.  As a consequence, there is no known combinatorial description of the broken lines on these scattering diagrams, though some progress has been made \cite{akagi2023explicit,ENS,Rea}.  Conversely, the compatible pairs have a simple combinatorial description, so constructing an explicit bijection would yield the same for the broken lines.

The structure of the paper is as follows.  In \autoref{sec: prelim cp}, we provide preliminaries concerning compatible pairs on maximal Dyck paths and Rupel's quantum grading.  We introduce the \emph{cascade} of a compatible pair in \autoref{sec: shadows and cascades} and relate the cascade to the previously-studied shadow of a compatible pair.  \autoref{sec: prelim bl} contains preliminaries on quantum scattering diagrams and broken lines.  In \autoref{sec: weights of bls}, we calculate the quantum weights of certain broken lines appearing in the theta basis of the quantum $r$-Kronecker cluster algebra.  We then construct a $q$-weighted bijection between positive compatible pairs and broken lines of negative angular momentum corresponding to (quantum) cluster monomials in \autoref{sec: bijection for r-kronecker}.  We extend this bijection to all broken lines appearing in the theta basis of the quantum Kronecker cluster algebra in \autoref{sec: bijection for kronecker}, where we also discuss the relation between several bases of this cluster algebra.

\section{Preliminaries: Compatible Pairs}\label{sec: prelim cp}

\subsection{Maximal Dyck Paths}
Fix $\ell,h \in \Z_{\geq 0}$. Consider a rectangle with vertices $(0,0)$, $(0,h)$, $(\ell,0)$, and $(\ell,h)$ with a main diagonal from $(0,0)$ to $(\ell,h)$.  

\begin{defn}
A \emph{Dyck path} is a lattice path in $\Z^2$ starting at $(0,0)$ and ending at a lattice point $(\ell,h)$ where $\ell,h \geq 0$, proceeding by only unit north and east steps and never passing strictly above the main diagonal. Given a collection $C$ of disjoint subpaths of a Dyck path, we denote the set of east steps by $C_1$, the set of north steps by $C_2$, and the total number of edges by $|C|$.  The \emph{length} of the Dyck path $\calP$ is the quantity $|\calP|$.  We denote the set of lattice points contained in the Dyck path $\calP$, ordered from left to right and including both endpoints, by $V(\calP) = \{w_0,w_1,\dots,w_{|\calP|}\}$.
\end{defn} 

The Dyck paths from $(0,0)$ to $(\ell,h)$ form a partially ordered set by comparing the heights at all vertices.  The \emph{maximal Dyck path} $\calP(\ell,h)$ is the maximal element under this partial order.

\begin{defn}\label{defn: maximal Dyck path}
For nonnegative integers $\ell$ and $h$, the \emph{maximal Dyck path} $\calP(\ell,h)$ is the path proceeding by unit north and east steps from $(0,0)$ to $(\ell,h)$ that is closest to the main diagonal without crossing strictly above it.
\end{defn} 

In the setting of combinatorics on words, maximal Dyck paths are also known as Christoffel words.  The maximal Dyck path $\calP(\ell,h)$ corresponds to the lower Christoffel word of slope $h/\ell$; see \cite{BLRS} for further details on Christoffel words.

Let the horizontal (\resp vertical) edges of $\calP = \calP(\ell,h)$ be labeled by $\eta_i$ for $1 \leq i \leq \ell$ (\resp $\nu_j$ for $1 \leq j \leq h$), with the indices increasing to the east (\resp north).  Given an edge $e$ in $\calP$, let $p_e$ denote the left endpoint of $e$ if $e$ is horizontal or the top endpoint of $e$ if $e$ is vertical.  For distinct edges $e,f$ in $\calP(\ell,h)$, let $\overrightarrow{ef}$ denote the subpath proceeding east from $p_e$ to $p_f$, continuing cyclically around $\calP(\ell,h)$ if $e$ is to the east of $f$.  Similarly, for distinct vertices $w_i,w_j \in V(\calP)$, let $\overrightarrow{w_iw_j}$ denote the subpath proceeding east from $w_i$ to $w_j$, continuing cyclically if needed.

In the framework of Lee--Li--Zelevinsky \cite{LLZ}, the cluster variables correspond to a family of maximal Dyck paths with a similar recursive structure.

\begin{defn}
Let $\{c_n\}_{n=0}^\infty$ be the sequence of integers defined recursively by:
$$c_0 = -1, c_{1} = 0,\; \text{ and } c_{n} = rc_{n-1} - c_{n-2} \text{ for } n > 1\,.$$
\end{defn}

For $n \geq 3$, the maximal Dyck path associated to the cluster variable $X_n$ is $\calC_n \colonequals \calP(c_{n-1},c_{n-2})$.  The Dyck paths corresponding to cluster monomials are those of the form $\calP(\alpha c_{n+1} + \beta c_{n}, \alpha c_n + \beta c_{n-1})$ for integers $n \geq 1$ and $\alpha,\beta \geq 0$.

\subsection{Compatible Pairs}
We now define compatible pairs, certain collections of edges on a Dyck path $\calP$, originally introduced in \cite{LLZ}.  
\begin{defn}
For any pair of vertices $u,w\in \calP(\ell,h)$,  let $|uw|_1$ (\resp $|uw|_2$) denote the number of horizontal (\resp vertical) edges of $\overrightarrow{uw}$.  Given a set of horizontal edges $S_1$ and a set of vertical edges $S_2$ in $\calP(\ell,h)$, the pair $(S_1,S_2)$ is \emph{compatible} if, for every edge $e$ in $S_1$ and every edge $f$ in $S_2$, there exists a lattice point $t \neq p_e,p_f$ in the subpath $\overrightarrow{e f}$ such that 
$$|tp_f|_1 = r|\overrightarrow{tp_f} \cap S_2| \text{ or } |p_e t|_2 = r|\overrightarrow{p_e t} \cap S_1|\,.$$
\end{defn}

The expansion formula for cluster variables given by Lee, Li, and Zelevinsky has monomials corresponding to compatible pairs on $\calC_n$.  Their expansion formula works in the more general setting of elements of the greedy basis, which contains the cluster variables.  For further details on the greedy basis, see \cite{LLZ}.  We present their formula in the special case of classical cluster variables $x_n$, which are a specialization of the quantum cluster variables $X_n$ obtained by setting $q = 1$.

\begin{thm}{\cite[Theorem 1.11]{LLZ}}\label{thm: LLZ expansion}
For each $n \geq 3$, the (classical) cluster variable $x_n$ in $\calA(r,r)$ is given by
$$x_n = x_1^{-c_{n-1}}x_2^{-c_{n-2}} \sum_{(S_1,S_2)}x_1^{r|S_2|}x_2^{r|S_1|}\,,$$
where the sum is over all compatible pairs $(S_1,S_2)$ in $\calC_n$.
\end{thm}

Let $\CP(\calP)$ denote the set of all compatible pairs on $\calP$.  Let $\CP(\ell,h,a,b)$ be the set of pairs $(S_1,S_2) \in \CP(\calP(\ell,h))$ such that $|S_1| = a$ and $|S_2| = b$.  In the cluster variable case, $\CP(c_{n-1},c_{n-2},a,b)$ is the set of compatible pairs corresponding to the monomial $x_1^{rb-c_{n-1}}x_2^{ra-c_{n-2}}$ in the Laurent polynomial expansion of $X_n$.  

\begin{defn} We say that a compatible pair in $\CP(\ell,h,a,b)$ is \emph{positive} if $rb \leq \ell$.
\end{defn}

The class of positive compatible pairs arose naturally in Lee and Schiffler's proof of the positivity conjecture for cluster algebras from quivers (see, for example, the first term in \cite[Theorem 3.22]{LSpositive}).  In this paper, we will primarily be focused on positive compatible pairs on maximal Dyck paths corresponding to cluster monomials.  We now show that on these Dyck paths, the positive compatible pairs comprise almost all compatible pairs.

\begin{lem}\label{lem: almost all}
Fix $\alpha,\beta \in \Z_{\geq 0}$.  Let $\calP_n$ be the maximal Dyck path corresponding to the cluster monomial $X_n^\alpha X_{n+1}^\beta$ in $\calA(r,r)$ for $r \geq 3$ and let $\CP_{+}(\calP_n)$ be the set of compatible pairs on $\calP_n$ that are positive.  Then $$\lim_{n \to \infty} \frac{|\CP_{+}(\calP_n)|}{|\CP(\calP_n)|} = 1\,.$$
\end{lem}
\begin{proof}
We will only explicitly handle the cluster variable case, but the cluster monomial case follows from an analogous argument.  Let $\calP = \calP(c_{n+1},c_{n})$, and let $\CP_{\text{bad}}(\calP) = \CP(\calP) \cut \CP_{+}(\calP)$. By considering, for each integer $j \in \{0,1,\dots,c_{n-2}\}$, the compatible pairs with $|S_2| = b \leq c_{n-2} \leq \frac{1}{r}c_{n+1}$ and $S_1$ disjoint from $\sh(S_2)$ (see \autoref{sec: shadows and cascades} for this construction), we obtain that 
$$|\CP(\calP)| \geq \sum_{b = 0}^{c_{n-2}} \binom{c_n}{b} 2^{c_{n+1} - rb} \geq 2^{c_{n+1} - rc_{n-2}}\sum_{k = 0}^{c_{n-2}} \binom{c_n}{k} \,.$$
If $rb > c_{n+1}$, then $b \geq c_n - c_{n-2}$. Each vertical edge in $S_2$ has at least $(r-1)$ horizontal edges immediately preceding it that cannot be in $S_1$.  Thus, we have
$$|\CP_{\text{bad}}(\calP)| \leq \sum_{b = c_n - c_{n-2}}^{c_n} \binom{c_n}{b} 2^{c_{n+1} - (r-1)b} \leq 2^{c_{n+1} - (r-1)(c_n - c_{n-2})}\sum_{k = 0}^{c_{n-2}} \binom{c_n}{k} \,.$$

Since $c_k \geq (r-1)c_{k-1}$ for each $k$, we have 
$$(r-1)(c_n-c_{n-2}) - rc_{n-2} \geq ((r-1)^3 -2r + 1)c_{n-2} \geq c_{n-2}\,.$$
We can therefore see that
$$ 1- \frac{|\CP_{+}(\calP)|}{|\CP(\calP)|} = \frac{|\CP_{\text{bad}}(\calP)|}{|\CP(\calP)|}\leq 2^{-c_{n-2}}\,,$$
which goes to $0$ as $n$ approaches infinity.
\end{proof}

\subsection{Quantum Weighting on Compatible Pairs}
In \cite[Corollary 5.7]{Rup2}, Rupel gives a quantum weighting to \emph{compatible gradings} \cite[Definition 1.2]{Rup2} corresponding to a generalization of quantum cluster variables.  In the case of quantum cluster variables, the corresponding compatible gradings can be viewed as compatible pairs.  

 The quantum cluster algebra we work with is the principal quantization of the rank-$2$ cluster algebra associated to the $r$-Kronecker quiver, which consists of two vertices with $r$ arrows between them.  While there are several choices for quantizing cluster algebras, we focus on the unique choice that is \emph{bar-invariant}, i.e., invariant under the \emph{bar-involution} $\bar f(q) \colonequals f(q^{-1})$ for $f \in \Z[q^{\pm 1}]$ and $\overline{fX_1^{a_1}X_2^{a_2}} \colonequals \overline{f}X_2^{a_2}X_1^{a_1}$ for $a_1,a_2 \in \Z$.  The sequence of quantum cluster variables is periodic when $r = 1$, and otherwise all $X_n$ are distinct.  Whenever it is clear from context, we shorten ``quantum cluster variable'' to ``cluster variable'' and ``quantum cluster monomial'' to  ``cluster monomial''.

For ease of computation later, we translate each compatible pair into a finite word, following \cite{Bur}, so that we can utilize the language of combinatorics on words.  We work over the alphabet $A = \{h,v,H,V\}$ and let $A^*$ denote the set of finite words on $A$.  Each compatible pair corresponds to a word in  $A^*$ by reading the edges from bottom left to top right.  The letters $h$ and $H$ (\resp $v$ and $V$) represent horizontal (\resp vertical) edges, with the capital letter denoting those edges in $S_1$ (\resp $S_2$).  

We now describe Rupel's construction of a quantum weighting for compatible gradings, though only in the specialization to quantum cluster monomials.  Viewing compatible pairs as words, this weighting takes the form of a morphism $w_q: \Z A^* \to \Z$, where $\Z A^*$ is the group of formal $\Z$-sums of words in $A^*$.

The function $w_q$ is defined on words of length $2$ in $A^*$ by:
$$w_q(hv) = w_q(Hv) = w_q(hV) = 1\,, \;\;\; w_q(Hh) = w_q(vV) = r\,, \;\;\; w_q(VH) = r^2 - 1\,,$$
and for $x,y \in A$, we set $w_q(xy) = - w_q(yx)$.  This last conditions implies that we have ${w_q(hh) = w_q(HH) = w_q(vv) = w_q(VV) = 0}$.  For a word $\sigma = \sigma_1\sigma_2\cdots\sigma_\ell \in A^*$, where each $\sigma_i$ is a letter in $A$, we set
$$w_q(\sigma) \colonequals \sum_{1 \leq i < j \leq \ell} w_q(\sigma_i\sigma_j)\,.$$
We then extend $w_q$ additively to formal $\Z$-sums of any words on $A$.  By considering the word in $A^*$ corresponding to a compatible pair, we also allow $w_q$ to be applied to compatible pairs.

\begin{rem}
    Note that the quantity $w_q(S_1,S_2)$ corresponds to the quantity $\beta_\omega + \gamma_\omega$ in Rupel's work \cite[Corollary 5.7]{Rup2}.
\end{rem}

This quantity $w_q(S_1,S_2)$ is the \emph{quantum weight} of the compatible pair $(S_1,S_2)$, as constructed by Rupel \cite{Rup2}. This allows us to calculate the Laurent expansion of quantum cluster variables as follows.

\begin{thm}[{\cite[Corollary 5.7]{Rup2}}]\label{thm: rupel quantum expansion}
 Consider the quantum cluster algebra $\calA_q(r,r)$ with quantum cluster variables $X_i$ for $i \in \Z$.  For $n \geq 4$, we have
 $$X_n = \sum_{(S_1,S_2)} q^{1 - c_{n-1} - c_{n-2} + w_q(S_1,S_2)}X_1^{-c_{n-1} + r|S_2|}X_2^{-c_{n-2} + r|S_1|}$$
 and 
 $$X_{3-n} = \sum_{(S_1,S_2)} q^{1 - c_{n-1} - c_{n-2} + w_q(S_1,S_2)}X_2^{-c_{n-1} + r|S_2|}X_1^{-c_{n-2} + r|S_1|}\,,$$
 where both sums range over all compatible pairs on $\calC_n$.
\end{thm}

Note that the shift factor of $q^{1 - c_{n-1} - c_{n-2}}$ must be included to have the resulting expression be bar-invariant.

\section{Shadows and Cascades in Compatible 
Pairs}\label{sec: shadows and cascades}
We begin by recalling some results about shadows of edges in compatible pairs, as defined by Lee--Li--Zelevinsky.  We then introduce a new notion called the \emph{cascade} of a compatible pair and establish connections between the two notions.

\subsection{Shadows}
In their study of compatible pairs, Li, Lee, and Zelevinsky \cite{LLZ} introduced the notion of the ``shadow'' of a set of vertical edges.

\begin{defn}\label{def: compatible pair}
For a vertical edge $\nu \in S_2$ with upper endpoint $w$, we define its \emph{local shadow}, denoted $\sh(\nu;S_2)$, to be the set of horizontal edges in the shortest subpath $\overrightarrow{\eta\nu}$ of $\calP$ such that $|\eta\nu|_1 = r|\overrightarrow{\eta\nu} \cap S_2|$.  In this case, we say that the edges $\eta$ and $\nu$ are \emph{shadow-paired} with each other.  If there is no such subpath $\overrightarrow{\eta\nu}$, then we define the local shadow to be $\calP_1$.

 For $V \subseteq S_2$, let $\displaystyle \sh(V;S_2) = \bigcup_{\nu \in V} \sh(\nu;S_2)$, and write $\sh(S_2) := \sh(S_2;S_2)$.  An edge $\eta \in \calP_1$ is called \emph{shadowed} if it is in $\sh(S_2)$, and we say that the edge $\nu \in S_2$ \emph{shadows} each edge in $\sh(\nu;S_2)$.
\end{defn}

\begin{defn}
We say that an edge $\eta \in \calP_1$ is \emph{right-shadowed} if it is in the local shadow $\sh(\nu;S_2)$ of some $\nu \in S_2$ where $\nu$ is to the right of $\eta$.  An edge that is shadowed but not right-shadowed is called \emph{left-shadowed}.  
\end{defn}

\subsection{Cascades}
In this section, we construct the \emph{cascade} of the vertical edges in a compatible pair.  This construction is crucial to the definition of the bijection between broken lines and compatible pairs given in \autoref{sec: bijection for r-kronecker} and \autoref{sec: bijection for kronecker}. We then draw connections between shadows in compatible pairs arising in the study of $\calA_q(r,r)$ and cascades.  

For each positive integer $m$, we associate $m$ horizontal edges to each vertical edge $\nu_j$, where the $i^{\text{th}}$ horizontal edge associated to $\nu_j$ is assigned the label $\nu_j^{(i)}$.  We then consider the set of labels 
$$\mathcal{P}_2^{(m)}=\{\nu_j^{(i)} \, : \, j\in [1,h]\text{ and }i\in [1,m]\}\,.$$

\begin{figure}
\begin{tikzpicture}[scale=.4]
\draw[step=1,color=gray] (0,0) grid (16,6);
\draw[line width=1,color=black] (0,0)--(3,0)--(3,1)--(6,1)--(6,2)--(8,2)--(8,3)--(11,3)--(11,4)--(14,4)--(14,5)--(16,5)--(16,6);
\draw[line width=2pt,color=blue]    (6,1.5) to[out=180,in=90,distance=0.8cm] (3.5,1);
\draw[line width=2pt,color=blue]    (11,3.5) to[out=180,in=90,distance=0.8cm] (8.5,3);
\draw[line width=2pt,color=blue]    (14,4.5) to[out=180,in=90,distance=0.8cm] (11.5,4);
\draw[line width=2pt,color=blue]    (16,5.5) to[out=180,in=90,distance=5cm] (1.5,0);
\draw[line width=2pt,color=blue]    (8,2.5) to[out=180,in=90,distance=3cm] (2.5,0);
\draw[line width=2pt,color=red] (6,1)--(6,2);
\draw[line width=2pt,color=red] (8,2)--(8,3);
\draw[line width=2pt,color=red] (11,3)--(11,4);
\draw[line width=2pt,color=red] (14,4)--(14,5);
\draw[line width=2pt,color=red] (16,5)--(16,6);
\end{tikzpicture}\qquad
\begin{tikzpicture}[scale=.2]
\draw[step=1,color=gray] (0,0) grid (24,23);
\draw[line width=1,color=black] (0,0)--(2,0)--(2,1)--(3,1)--(3,2)--(4,2)--(4,3)--(5,3)--(5,4)--(6,4)--(6,5)--(7,5)--(7,6)--(8,6)--(8,7)--(9,7)--(9,8)--(10,8)--(10,9)--(11,9)--(11,10)--(12,10)--(12,11)--(13,11)--(13,12)--(14,12)--(14,13)--(15,13)--(15,14)--(16,14)--(16,15)--(17,15)--(17,16)--(18,16)--(18,17)--(19,17)--(19,18)--(20,18)--(20,19)--(21,19)--(21,20)--(22,20)--(22,21)--(23,21)--(23,22)--(24,22)--(24,23);
\draw[line width=2pt,color=blue]    (9,7.5) to[out=170,in=100,distance=1.7cm] (9.5,8);
\draw[line width=2pt,color=blue]    (8,6.5) to[out=170,in=100,distance=1.5cm] (6.5,5);
\draw[line width=2pt,color=blue]    (5,3.5) to[out=170,in=100,distance=1.5cm] (3.5,2);
\draw[line width=2pt,color=blue]    (2,0.5) to[out=170,in=100,distance=0.7cm] (0.5,0);
\draw[line width=2pt,color=blue]    (3,1.5) to[out=110,in=160,distance=5.5cm] (11.5,10);
\draw[line width=2pt,color=blue]    (6,4.5) to[out=110,in=160,distance=3.5cm] (10.5,9);
\draw[line width=2pt,color=blue]    (17,15.5) to[out=170,in=100,distance=1.5cm] (15.5,14);
\draw[line width=2pt,color=blue]    (19,17.5) to[out=170,in=100,distance=1.5cm] (17.5,16);
\draw[line width=2pt,color=blue]    (22,20.5) to[out=170,in=100,distance=1.5cm] (20.5,19);
\draw[line width=2pt,color=blue]    (20,18.5) to[out=170,in=100,distance=3cm] (14.5,13);
\draw[line width=2pt,color=blue]    (23,21.5) to[out=170,in=100,distance=5.5cm] (13.5,12);
\draw[line width=2pt,color=blue]    (24,22.5) to[out=170,in=100,distance=7cm] (12.5,11);
\draw[line width=2pt,color=red] (2,0)--(2,1);
\draw[line width=2pt,color=red] (3,1)--(3,2);
\draw[line width=2pt,color=red] (5,3)--(5,4);
\draw[line width=2pt,color=red] (6,4)--(6,5);
\draw[line width=2pt,color=red] (8,6)--(8,7);
\draw[line width=2pt,color=red] (9,7)--(9,8);
\draw[line width=2pt,color=red] (17,15)--(17,16);
\draw[line width=2pt,color=red] (19,17)--(19,18);
\draw[line width=2pt,color=red] (20,18)--(20,19);
\draw[line width=2pt,color=red] (22,20)--(22,21);
\draw[line width=2pt,color=red] (23,21)--(23,22);
\draw[line width=2pt,color=red] (24,22)--(24,23);
\end{tikzpicture}
\caption{In the left image, we let $S_2$ be the set of red vertical edges on $\calP(16,6)$, i.e., $\{\nu_i : 2 \leq i \leq 6\}$. In the right image, we let $S_2$ be the set of red vertical edges on $\calP(24,23)$, i.e., $\{\nu_1,\nu_2,\nu_4,\nu_5,\nu_7,\nu_8,\nu_{16},\nu_{18},\nu_{19},\nu_{21},\nu_{22},\nu_{23}\}$.  In each image, every edge $\nu_i$ in $S_2$ is connected by a blue arc to the horizontal edge that is cascade-paired with $\nu_i$ (where $r = 3$ in the left image and $r = 2$ in the right image).  Note that any horizontal edge that lies below an arc but is not connected to any arc belongs to $\cas_m(S_2)$ for some $1 \leq m < r$. }
\label{fig: cascade example}
\end{figure}
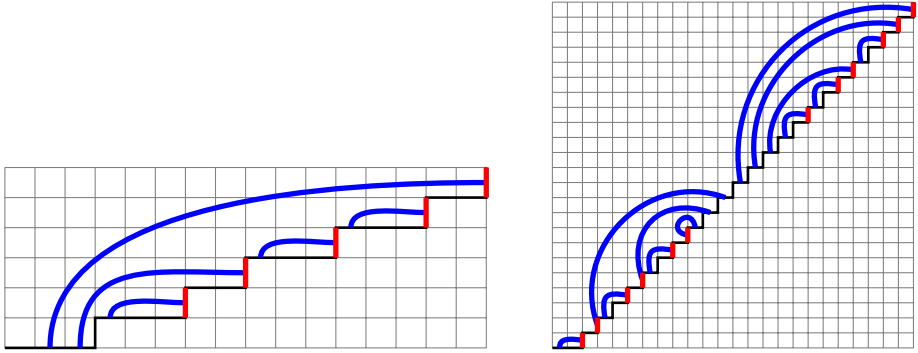

\begin{defn}\label{defn: cascade}\blfootnote{The name \emph{cascade} comes from the following visualization: View the path $\calP$ along with the vertical ray emanating from its leftmost vertex as the walls of a 2-dimensional fountain.  Each edge in $S_2$ has a stream of water emanating from its left side that flows into a horizontal edge in $\calP_2 \cut \left(\bigcup_{m=1}^{r-1} \cas_m(S_2) \right)$. Naturally, streams flow downward when possible. Those streams which cannot flow downward instead ``increase the water level'' of the ``pool'' at the bottom of the cascade by $1$.  }
Given $S_2 \subset \calP_2$ such that $r|S_2| \leq |\calP_1|$, we construct the \emph{cascade  of $S_2$}, denoted by $\cas(S_2)$, as follows.  Let $V$ be the list of edges in $S_2$ ordered from bottom to top.  We will modify a list $L$ of elements from $\calP_2^{(r)}$, with $L$ initially being empty.
\begin{enumerate}[(1)]
\item\label{step: new vertical} Remove the first element, $\nu_j$, from $V$.  Add $r$ labels $\nu^{(1)}_{j}, \nu^{(2)}_{j}, \dots, \nu^{(r)}_{j} \in \calP_2^{(r)}$ to the front of the list $L$. 
\item\label{step: pair left} If there are any unlabeled horizontal edges to the left of $\nu_j$, label the rightmost such edge with the first label in $L$ and remove this element from $L$.  Repeat until there are no such unlabeled horizontal edges or $L$ is empty.  Go back to Step \ref{step: new vertical} unless $V$ is empty.
\item\label{step: pair right} If all horizontal edges to the left of $\nu_j$ are labeled and $L$ is nonempty, label the leftmost unlabeled horizontal edge by the first element of $L$, removing this label from $L$.  Repeat until $L$ is empty or there are no unlabeled horizontal edges.
\end{enumerate}

We then set the \emph{cascade} $\cas(S_2)$ to be the set of edges in $H$ that are paired with some edge in $S_2$ via the above process.  For $m \leq r$, we say that the horizontal edge labeled by $\nu_j^{(m)}$ via the above process is  \emph{$m$-cascade-paired} with $\nu_j \in S_2$, and we let $\cas_m(S_2)$ denote the set of horizontal edges that are $m$-cascade-paired with some edge of $S_2$.  When $m = r$, we omit the prefix $m-$ and simply say that the edge labeled by $\nu_j^{(r)}$ is  \emph{cascade-paired} with $\nu_j \in S_2$.  

Any horizontal edge $m$-cascade-paired, for some choice of $m$, with $\nu_j$ is to its left (\resp right) in $\cas(S_2)$ is called \emph{left-filled} (\resp \emph{right-filled}).  Any horizontal edge that is not paired in $\cas(S_2)$ is called \emph{unfilled}.  For any edge $\nu \in S_2$, the \emph{local cascade of $\nu$}, denoted by $\cas(\nu;S_2)$, is the set of horizontal edges between the leftmost and rightmost horizontal edges paired with $\nu_j$ (including these edges).  
\end{defn}

\begin{rem}
As one may note in \autoref{fig: cascade example}, arcs can be drawn between each $\nu_i \in S_2$ and the horizontal edge cascade-paired to it such that all arcs lie above the path $\calP$ and are non-crossing.  This is in fact true for any set $S_2 \subset \calP_2$ where $r|S_2| \leq |\calP_1|$.  It is straightforward to show this by induction on the size of $S_2$.  
\end{rem}

\begin{prop}\label{prop: r cascade subset}
    Suppose $(r-1)h \leq \ell \leq rh$.  Then for any $(S_1,S_2) \in \CP_{\cas}(\ell,h)$ such that $r|S_2| \leq \ell$, we have $S_1 \subseteq \cas_r(S_2)$.  
\end{prop}
\begin{proof}
Each vertical edge in $\calP(\ell,h)$ is immediately preceded by at least $r-1$ horizontal edges.  Thus, $\bigcup_{i=1}^{r-1}\cas_{i}(S_2)$ consists of the $r-1$ horizontal edges preceding each edge in $S_2$.  The (at most $r$) horizontal edges immediately preceding an edge in $S_2$ cannot be in $S_1$ by the compatibility restrictions.  Therefore, we have $S_1 \subseteq \cas(S_2) \cut \left(\bigcup_{i=1}^{r-1}\cas_{i}(S_2)\right) = \cas_r(S_2)$.
\end{proof}

\subsection{Relationship between Cascades and Shadows}
The cascade of a set of vertical edges is equal to the size of its shadow.  In this section, we construct a bijection between these two sets of edges that respects compatibility.  This allows us to later adapt some results of Lee--Li--Zelevinsky on shadows to the setting of cascades, which is necessary for the resulting quantum weights to agree with Rupel's quantum weighting.

Let $\CP_{\sh}(\ell,h,a,b)$ be the set of pairs $(S_1,S_2)$ in $\CP(\ell,h,a,b)$ such that $S_1 \subseteq \sh(S_2)$.  Similarly, let $\CP_{\cas}(\ell,h,a,b)$ be the set of pairs $(S_1,S_2)$ in $\CP(\ell,h,a,b)$ such that $S_1 \subseteq \cas(S_2)$.

\begin{defn}\label{defn: lambda}
    Let $k = \max\{m \in [\ell + h] : |w_0 w_m|_1 = r|\vec{w_0 w_m} \cap S_2|\}$. Let $\eta_{i_1},\dots,\eta_{i_\ell}$ be the list of left-filled edges and $\eta_{j_1},\dots,\eta_{j_\ell}$ be the list of left-shadowed edges in $\calP_1$, where both lists are ordered from left to right.  
    
    We define an involution $\iota: \calP_1 \to \calP_1$ that swaps each $\eta_{i_s}$ with $\eta_{j_s}$ and leaves all remaining edges unchanged. 
    
    We then define the map $\lambda: \CP(\ell,h,a,b) \to \CP(\ell,h,a,b)$ taking a compatible pair $(S_1,S_2)$ to 
    $$\lambda(S_1,S_2) \colonequals (\iota(S_1),S_2), \text{ where } \iota(S_1) = \{\iota(\eta) \in \calP_1 : \eta \in S_1\}\,.$$
\end{defn}

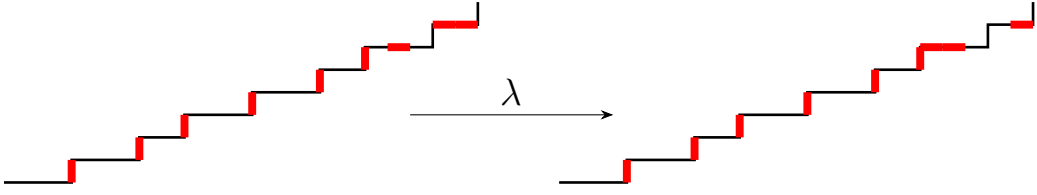
\begin{figure}[h]
\begin{tikzpicture}[scale=.3]
\draw[line width=1,color=black] (0,0)--(3,0)--(3,1)--(6,1)--(6,2)--(8,2)--(8,3)--(11,3)--(11,4)--(14,4)--(14,5)--(16,5)--(16,6)--(19,6)--(19,7)--(21,7)--(21,8);
\draw[line width=3pt,color=red] (3,0)--(3,1);
\draw[line width=3pt,color=red] (6,1)--(6,2);
\draw[line width=3pt,color=red] (8,2)--(8,3);
\draw[line width=3pt,color=red] (11,3)--(11,4);
\draw[line width=3pt,color=red] (14,4)--(14,5);
\draw[line width=3pt,color=red] (16,5)--(16,6);
\draw[line width=3pt,color=red] (17,6)--(18,6);
\draw[line width=3pt,color=red] (19,7)--(20,7);
\draw[line width=3pt,color=red] (20,7)--(21,7);
\draw[line width=1,color=black,xshift=700] (0,0)--(3,0)--(3,1)--(6,1)--(6,2)--(8,2)--(8,3)--(11,3)--(11,4)--(14,4)--(14,5)--(16,5)--(16,6)--(19,6)--(19,7)--(21,7)--(21,8);
\draw[line width=3pt,color=red,xshift=700] (3,0)--(3,1);
\draw[line width=3pt,color=red,xshift=700] (6,1)--(6,2);
\draw[line width=3pt,color=red,xshift=700] (8,2)--(8,3);
\draw[line width=3pt,color=red,xshift=700] (11,3)--(11,4);
\draw[line width=3pt,color=red,xshift=700] (14,4)--(14,5);
\draw[line width=3pt,color=red,xshift=700] (16,5)--(16,6);
\draw[line width=3pt,color=red,xshift=700] (17,6)--(18,6);
\draw[line width=3pt,color=red,xshift=700] (16,6)--(17,6);
\draw[line width=3pt,color=red,xshift=700] (20,7)--(21,7);
\draw (21.5,4) node[anchor=west]  {\large$\lambda$};
\draw[-Stealth] (18,3)--(27,3);
\end{tikzpicture}
\caption{An illustration of the map $\lambda$ applied to the compatible pair $(S_1,S_2)$ in $\calP(21,8)$, where $r = 3$, $S_1 = \{\eta_{18},\eta_{20},\eta_{21}\}$, and $S_2 = \{\nu_1,\nu_2,\nu_3,\nu_4,\nu_5,\nu_6\}$. Following \autoref{defn: lambda}, we have $(i_1,i_2) = (17,18)$ and $(j_1,j_2) = (20,21)$.  We thus have $\lambda(S_1) = \{\eta_{17},\eta_{18},\eta_{21}\}$.  The compatible pair $\lambda(S_1,S_2) = (\lambda(S_1),S_2)$ is depicted on the right.}
\label{fig: lambda exmp}
\end{figure}

\begin{lem}\label{lem: move edge left}
Suppose $(S_1,S_2)$ is a compatible pair.  Let $\eta_i \notin S_1$ and $\eta_j\in S_1$ be such that any edge in $S_2$ that shadows $\eta_i$ also shadows $\eta_j$, $i < j$, and $\eta_i$ and $\eta_j$ are either both left-shadowed or both right-shadowed. Then $((S_1 \cup \{\eta_i\}) \cut \{\eta_j\}, S_2)$ is also a compatible pair.
\end{lem}
\begin{proof}
It is enough to show that the compatibility condition holds for any vertical edge $\nu \in S_2$ such that $\eta_j$ is in $\sh(\nu;S_2)$ (and hence so is $\eta_i$).  Then we have $\overrightarrow{\eta_j \nu} \subseteq \overrightarrow{\eta_i\nu}$. 

Since $\nu$ shadows $\eta_j$, then $|e\nu|_1 < r|\overrightarrow{e \nu} \cap S_2|$ for all $e \in \overrightarrow{\eta_j \nu}$.  By the definition of compatibility, there must be some $e \in \overrightarrow{\eta_j \nu}$ such that $|\eta_j e|_2 = r|\overrightarrow{\eta_j e} \cap S_1|$. 

Let $\eta_k$ be the leftmost horizontal edge $\overrightarrow{\eta_i \eta_j} \cap S_1$.  We then have that $|\eta_i \nu|_2 = |\eta_i \eta_k|_2 + |\eta_k \nu|_2$.  We already have that $r|\vec{\eta_k \nu} \cap S_1| \leq |\eta_k \nu|_2$.  We therefore have 
$$r|\vec{\eta_i \nu }\cap (S_1 \cut \eta_j)| = r|\vec{\eta_k \nu}\cap S_1| \leq |\eta_k \nu|_2 \leq |\eta_i \nu|_2\,,$$
as desired.
\end{proof}

\begin{lem}
The map $\lambda$ preserves compatibility on any compatible pair.
\end{lem}
\begin{proof}
The map $\lambda$ can be obtained by successive iterations of the operation in \autoref{lem: move edge left}, which preserves compatibility.  In particular, we replace each edge $e'$ in $\sh(S_2) \cap S_1$ with its corresponding edge $e$ in $\cas(S_2)$.  Either the edges $e$ and $e'$ coincide, $e$ is not shadowed, or they satisfy the conditions of \autoref{lem: move edge left}.
\end{proof}

\begin{lem}\label{lem: lambda inverse compatibility}
For compatible pairs on Dyck paths corresponding to cluster monomials, $\lambda^{-1}$ preserves compatibility.
\end{lem}
\begin{proof}
For cluster monomial Dyck paths, compatibility only needs to be considered on a single path (without wrapping around cyclically) according to \cite[Remark 2.21]{Rup4}. Since $\lambda$ only involves edges that are not right-shadowed, any subset of these can be included while preserving compatibility.
\end{proof}

\begin{cor}
Let $\calP(\ell,h)$ correspond to a cluster monomial.  Then for any $a,b \in \Z_{\geq 0}$ such that $rb \leq \ell$, we have
$$|\CP_{\sh}(\ell,h,a,b)| = |\CP_{\cas}(\ell,h,a,b)|\,.$$
\end{cor}

\begin{lem}\label{lem: lambda shadow exchange}
Let $(\lambda(S_1),S_2) = \lambda(S_1,S_2)$.  Then a vertical edge $\nu \in S_2$ is shadowing an edge in $S_1$ if and only $\nu$ is filling an edge in $\lambda(S_1)$. 
\end{lem}

\begin{rem}
The definition of shadow and remote shadow in \cite{LLZ} does not seem to be compatible with Rupel's quantum grading.  Thus, we introduce the cascade framework in order for the resulting partition of $S_2$ to yield the correct set of quantum weights via Rupel's assignment.  However, we could alternatively have used the original shadow definitions from \cite{LLZ} and instead introduced a new quantum grading.
\end{rem}

\section{Preliminaries: Scattering Diagrams and Broken Lines}\label{sec: prelim bl}
While we provide the necessary details to construct rank-$2$ cluster scattering diagrams, we refer to \cite[Section 2]{davison2021strong} for full details on quantum scattering diagrams. We then define broken lines on these scattering diagrams, including how to assign quantum weights to broken lines crossing over cluster walls.  

\subsection{Scattering Diagrams}\label{subsec: broken lines}
 Fix a rank-$2$ lattice $M \cong \Z^2$. Let $M_\R = M \otimes \R$, and for a strictly convex rational cone $\sigma \subsetneq M_\R$, let $P = P_\sigma = \sigma \cap M$.  We set $\widehat{\Z[P]}$ to be the completion of the monoid ring $\Z[P]$ at the maximal monomial ideal $\frakm$ generated by $\{x^m : m \in P \cut \{0\}\}$.
\begin{defn}
A \emph{wall} is a pair $(\frakd, f_\frakd)$ consisting of a support $\frakd \subseteq M_\R$ and an associated \emph{wall-crossing function} $f_\frakd \in \widehat{\Z[P_{\frakd}]}$, where
\begin{itemize}
    \item the support $\frakd$ is either a ray $\R_{\leq 0}w$ or a line $\R w$ for some $w \in \sigma \cap (M \cut \{0\})$;
    \item we have $\displaystyle f_\frakd = f_\frakd(x^w) = 1 + \sum_{k \geq 1} c_k x^{kw}$ for some $c_k \in \Z$.
\end{itemize}
\end{defn}

\begin{defn}
A \emph{scattering diagram} $\frakD$ is a collection of walls such that
$$\{(\frakd, f_\frakd) \in \frakD : f_\frakd \not \equiv 1 \mod \frakm^k\}$$
is finite for each $k \geq 0$.  The union of the supports of the walls is the \emph{support} $\Supp(\frakD)$ of the scattering diagram $\frakD$.  
\end{defn}

We associate a scattering diagram $\frakD_r$ to the cluster algebra $\calA(r,r)$.  Let the ``initial'' scattering diagram for $\calA(r,r)$ be 
$$\frakD_{r,\init} = ((\R(1,0), 1 + x_1^r), (\R(0,1), 1 + x_2^r))\,.$$
The scattering diagram $\frakD_r$ is then the \emph{consistent} scattering diagram formed by adding rays to $\frakD_{r,\init}$ (see, for example, \cite[Section 3]{CGM} for details on consistency).  It is shown in \cite[Theorem 1.7]{GHKK} that such a scattering diagram exists and is unique up to equivalence.  The scattering diagram $\frakD_r$ contains \emph{cluster walls} of the form 
$$(\R_{\leq 0}(c_m,c_{m-1}), 1 + x_1^{rc_m}x_2^{rc_{m-1}}) \text{ and } (\R_{\leq 0}(c_{m-1},c_{m}), 1 + x_1^{rc_{m-1}}x_2^{rc_{m}})$$
for all $m \geq 2$, and these walls are distinct for $r \geq 2$.  

For $r \geq 2$, there are walls in addition to the cluster walls that lie in the closed cone spanned by $(2r,-r^2 \pm r\sqrt{r^2 - 4})$.  This cone is known as the ``Badlands'' and contains infinitely many rays for $r \geq 3$.  It was recently shown by Davison and Mandel \cite{davison2021strong} that $\frakD_r$ has a wall at every rational slope within the Badlands. However, the functions associated to these walls are generally not understood, though some partial progress has been made \cite{akagi2023explicit, ENS, Rea}.

\begin{rem}
While scattering diagrams are usually parameterized in terms of ${\bf g}$-vectors, the setting of ${\bf d}$-vectors is better suited for drawing connections to compatible pairs.  Thus, our definition of the scattering diagram $\frakD_r$ is equivalent to the ${\bf d}$-vector scattering diagram $\frakD_{(r,r)}^{\bf d}$ in \cite{CGM}, and our theta functions are parameterized by ${\bf d}$-vectors as well.  The two settings differ by a piece-wise linear bijection on the scattering diagram \cite[Theorem 4.3]{CGM}.
\end{rem}

\subsection{The Kronecker Scattering Diagram}
We now describe a more explicit construction of the scattering diagram for the Kronecker cluster algebra. 

In this case, we have
$$\frakD_{2,\init } = \left\{ \left(\R(-1,0), 1 + x_1^2\right), \left(\R(0,1),1 + x_2^2\right)\right\}\,.$$
The remaining walls of slope $\neq 1$ can be obtained via an iterative process, described in \cite[Section 3]{CGM}.  This process yields the walls 
$$\frakD'_{2} = \bigcup_{\ell \geq 1}\left\{ \left(\R_{\leq 0}(\ell + 1, \ell), 1 + x_1^{2(\ell+1)} x_2^{2\ell}\right), \left(\R_{\leq 0}(\ell,\ell + 1),1 + x_1^{2\ell} x_2^{2(\ell+1)}\right)\right\}\,.$$

Lastly, there is one additional wall that is not described by this process, given by 
$$\frakd_{1} = \left(\R_{\leq 0}(1,1),\; \sum_{j = 0}^\infty (j+1)x_1^{2j}x_2^{2j}\right)\,.$$  Its support is a ray of slope $1$ eminating from the origin, and its wall-crossing function was originally computed in the classical case by Reineke \cite[Section 6]{Rei08} using representation theory.

The scattering diagram associated to the cluster algebra $\calA(2,2)$ is then given by 
$$\frakD_2= \frakD_{2,\init} \cup \frakD'_{2}  \cup \{\frakd_{1}\}\,.$$

We refer to the walls in $\calA(2,2)$ by their slope, i.e., 
\begin{align*}
    \frakd_{\ell/(\ell + 1)} &= \left(\R_{\leq 0}(\ell + 1, \ell), 1 + x_1^{2(\ell+1)} x_2^{2\ell}\right)\,,\\
    \frakd_{(\ell+1)/\ell} &= \left(\R_{\leq 0}(\ell,\ell + 1),1 + x_1^{2\ell} x_2^{2(\ell+1)}\right)\,,\\
    \frakd_0 &= \left(\R(1,0), 1 + x_1^2\right)\,, \text{ and }\\
    \frakd_\infty &= \left(\R(0,1),1 + x_2^2\right)\,.
\end{align*}

\begin{figure}
\begin{tikzpicture}[scale=.4]

\draw[line width=0.1pt] (-10,0)--(10,0);
\draw[line width=0.1pt] (0,-10)--(0,10);
\draw[line width=2pt,color=red] (0,0)--({atan(1)}:{-10 });
\foreach \x in {2,...,40}
    {
    \draw[line width=0.1pt] (0,0)--({atan(\x/(\x-1))}:{-10 });
    \draw[line width=0.1pt] (0,0)--({atan((\x-1)/\x)}:{-10 });
    }
\draw[line width=.1pt,color=red] (0,0)--({atan(1)}:{-10 });

\end{tikzpicture}
\caption{The walls of the scattering diagram $\frakD_2$ are depicted above.  The only non-cluster wall is $\frakd_1$, the wall of slope $1$, shown in red. }
\label{fig: D2}
\end{figure}
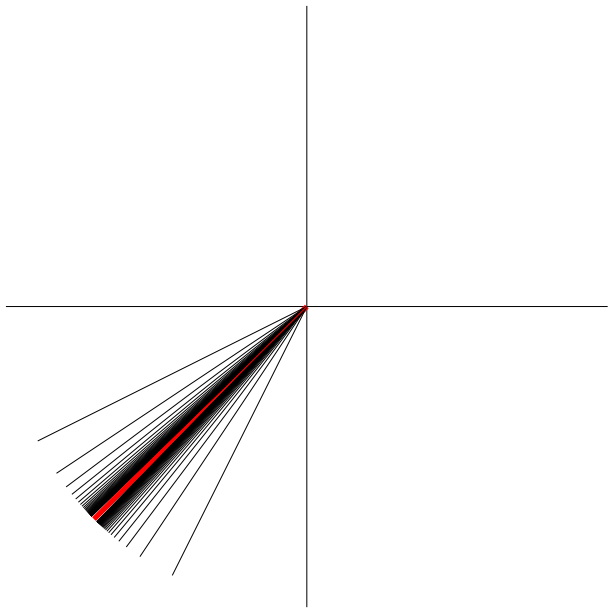

\subsection{Broken Lines and the Theta Basis}
We begin by introducing broken lines for classical scattering diagrams.  Then, focusing specifically on broken lines that only cross over cluster walls, we see how to assign a quantum weight to each wall crossing in a quantum cluster scattering diagram.  We additionally describe the quantum weighting associated with crossing the unique non-cluster wall in the quantum Kronecker scattering diagram.

\begin{defn}
    Let $\frakD$ be a consistent scattering diagram in $M_\R$.  Fix an initial exponent $v_0 \in M$ and terminal point $Q \in M_\R \cut \Supp(\frakD)$. A \emph{broken line} $\gamma$ on $\frakD$ with ends $(Q,m)$ is the data of a continuous map $\gamma: (-\infty,0] \to M_\R \cut \{0\}$, values $-\infty = \tau_0 <  \tau_1 \cdots < \tau_\ell = 0$, and for each $i = 1,\dots,\ell$, and associated monomial $c_iz^{v_i} \in \Z[M]$ such that
    \begin{enumerate}[(i)]
        \item $\gamma(0) = Q$ and $c_0 = 1$;
        \item $\gamma'(\tau) = -v_0$ for all $\tau < \tau_0$, and for each $i = 1,\dots,\ell$, we have $\gamma'(\tau) = -v_i$ for all $\tau \in (\tau_{i-1},\tau_i)$;
        \item for each $i = 1,\dots,\ell - 1$, the point $\gamma(\tau_i)$ is where $\gamma$ transversally crosses a wall $\frakd(i)$ of $\frakD$;
        \item $c_{i+1}z^{v_{i+1}}$ is a monomial term of 
        $$c_iz^{v_i} f^{v_i \cdot n}_{\frakd(i)}$$
        where $n$ is a primitive normal vector of $\frakd(i)$ such that $v_i \cdot n > 0$.
    \end{enumerate}
\end{defn}

The element $v_i \in \Z^2$ is called the \emph{exponent} of the domain of linearity $\gamma(\tau_i,\tau_{i+1})$.  We refer to each $\gamma(\tau_i)$ for $i = 1,\dots,\ell - 1$ as a \emph{bending of multiplicity $m_i$} of $\gamma$ at the wall $\frakd(i)$, where $v_{i+1}$ is the $(m_i+1)$-th term of $c_iz^{v_i} f^{v_i \cdot n}_{\frakd(i)}$ (ordered increasingly by exponent). 

The combinatorics of Gaussian binomial coefficients is necessary for discussing the quantum weight of broken lines.  For $k \in \Z_{\geq 0}$, we define
$$[k]_q  \colonequals \frac{q^k - q^{-k}}{q - q^{-1}} = q^{-(k-1)} + q^{-(k-3)} + \cdots + q^{k-3} + q^{k-1} \in \Z[q^{\pm 1}]\,.$$
This can be viewed as a quantum deformation of the integer $k$, since $\lim_{q \to 1} [k]_q = k$.  We then define $[k]_q! \colonequals [k]_q[k-1]_q\cdots[2]_q[1]_q$ for $k \geq 1$ and set $[0]_q! \colonequals 1$.  For integers $\ell \geq k \geq 0$, the \emph{bar-invariant quantum binomial coefficient} is the quantity
$$\binom{\ell}{k}_q \colonequals \frac{[\ell]_q!}{[k]_q! [\ell - k]_q!} \in \Z_{\geq 0}[q^{\pm 1}]\,.$$

\begin{defn}
Let $\gamma$ be a broken line on the quantum cluster scattering diagram $\frakD_r$ that bends only over cluster walls.  To each domain of linearity $\gamma(\tau_i,\tau_{i+1})$, we associate a \emph{quantum weight} $c_{i,q}$ determined by
$$c_{0,q} = 1 \text{ and } \displaystyle c_{i+1,q} = c_{i,q}\binom{v_i \cdot n}{m_i}_{q^{2r}}\,,$$
where $m_i$ is the multiplicity of the bending at $\gamma(\tau_{i+1})$.  The \emph{quantum weight} of $\gamma$ is the quantum weight of the last domain of linearity. 

If $\gamma$ is on the scattering diagram $\frakD_2$, we additionally allow bending on the non-cluster wall.  The quantum weight after this bending is given by $c_{i+1,q} = c_{i,q}f[m_i]$, where $f[m_i]$ is the coefficient of $x^{m_i}$ in 
$$f(x) = \left(\sum_{k=0}^\infty [i+1]_{q^4} x^k\right)^{v_i \cdot n}\,.$$
\end{defn}

For such broken lines, we associate a \emph{quantum monomial} $\Mono_q(\gamma) = c_{\ell,q}z^{v_\ell}$ to the last domain of linearity.

\begin{exmp}\label{exmp: bl q weight}
    The broken line $\gamma$ in \autoref{fig: broken line q weight} has quantum weight 
    \begin{align*}
        w_q(\gamma) &= \binom{(-12,-11)\cdot(2,-3)}{1}_{q^4}\binom{(-6,-7)\cdot(1,-2)}{1}_{q^4}\binom{(-2,-5)\cdot(-1,0)}{1}_{q^4}\\
    &= \binom{9}{1}_{q^4} \binom{8}{1}_{q^4}  \binom{2}{2}_{q^4}\\
        &= q^{-60} + 2 q^{-52} + 3 q^{-44} + 4 q^{-36} + 5 q^{-28} + 6 q^{-20} + 7 q^{-12} + 8 q^{-4} +  8 q^{4} + 7 q^{12}\\
    &\;\;\;\;\;\;\, + 6 q^{20} + 5 q^{28} + 4 q^{36} + 3 q^{44} + 2 q^{52} + q^{60}
    \end{align*}
\end{exmp}

\begin{rem}
We must renormalize the setup in \cite{davison2021strong} in order to have the quantum cluster variables coincide with those in the setting of \cite{Rup2}.  This renormalization is achieved by replacing the quantum coefficient by its $2r$-th power, i.e., applying the map $q \mapsto q^{2r}$.  
\end{rem}

\begin{rem}
While we have not explicitly defined the quantum cluster scattering diagrams, the quantum weight of a broken line in our setting matches the weight of the same broken line on the associated quantum cluster scattering diagram.  To handle bending at the non-cluster wall in $\frakD_2$, one can adapt the methods used by Reading \cite[Theorem 3.4]{Rea} to the quantum case. 
\end{rem}

\begin{figure}
\begin{tikzpicture}[scale=9]
\def\v{0.15cm}
\draw[line width=1pt] (-.4,0)--(.4,0);
\draw[line width=1pt] (0,-.4)--(0,.4);
\draw[line width=3pt,color=gray] (0,0)--({atan(1)}:{-.4 });
\foreach \x in {2,...,3}
    {
    \draw[line width=1pt] (0,0)--({atan((\x-1)/\x)}:{-.4 });
    }
\foreach \x in {4,...,30}
    {
    \draw[line width=0.2pt,color=gray,draw opacity=0.4] (0,0)--({atan((\x-1)/\x)}:{-.4 });
    }
\foreach \x in {2,...,30}
    {
    \draw[line width=0.1pt,color=gray,draw opacity=0.4] (0,0)--({atan(\x/(\x-1))}:{-.4 });
    }
\draw[line width=1pt, color=blue] (0.25,0.375) node[circle,fill=black,minimum size=\v,inner sep=0pt]{}--(0,0.25) node[circle,fill=blue,minimum size=\v,inner sep=0pt]{}--(-1/8,-1/16) node[circle,fill=blue,minimum size=\v,inner sep=0pt]{}--(-1/6,-1/9) node[circle,fill=blue,minimum size=\v,inner sep=0pt]{}--(-.3147,-0.2468);

\draw (-.54,-.25) node[anchor=west,color=blue]  {\small$(-12,-11)$};
\draw (-.37,-.07) node[anchor=west,color=blue]  {\small$(-6,-7)$};
\draw (-.25,0.12) node[anchor=west,color=blue]  {\small$(-2,-5)$};
\draw (0,0.37) node[anchor=west,color=blue]  {\small$(-2,-1)$};
\end{tikzpicture}
\caption{A broken line $\gamma$ with terminal point $(0.25,\,0.375)$ is depicted on the scattering diagram $\frakD_2$ (shown in \autoref{fig: D2}).   The exponent of each domain of linearity of $\gamma$ is shown as an element of $\Z^2$.  The quantum weight of $\gamma$ is computed in \autoref{exmp: bl q weight}. }
\label{fig: broken line q weight}
\end{figure}
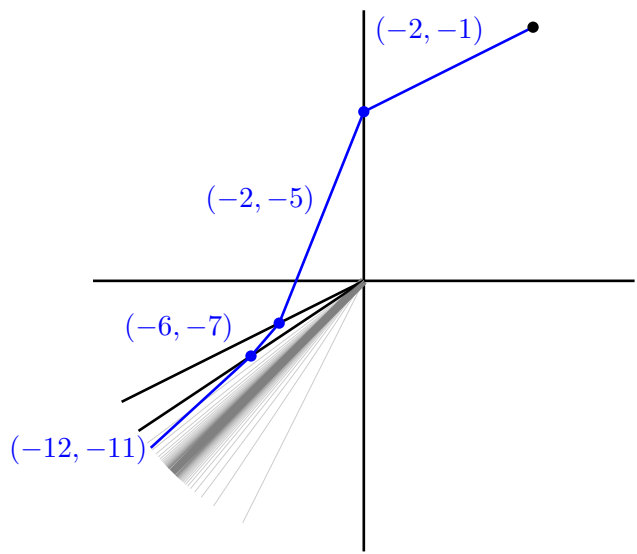

When all walls of $\frakD$ have apex at the origin, as is the case for scattering diagrams associated to cluster algebras, then we say $Q \in M_\R$ is \emph{generic} if $Q$ is not in $\bigcup_{v \in M^*} v^{\perp}$. 

\begin{defn}
Fix a scattering diagram $\frakD$, $m \in M \cut \{0\}$, and $Q \in M_\R \cut \Supp(\frakD)$.  The \emph{theta function} $\vartheta_{Q,m}$ is defined by
$$\vartheta_{Q,m} = \sum_{\gamma} \Mono(\gamma)$$
where the sum is taken over all broken lines that have endpoint $Q$ and initial exponent $m$.

Fixing a generic $Q \in \R^2_{\geq 0}$, the \emph{theta basis} is the set of all theta functions $\vartheta_{Q,m}$ for any $m \in \Z^2$.
\end{defn} 

Additionally, on the scattering diagram $\frakD_2$, we can define the \emph{quantum theta functions} and \emph{quantum theta basis} by replacing $\Mono(\gamma)$ with $\Mono_q(\gamma)$ in the above definition.  Quantum theta bases were defined for arbitrary cluster algebras in recent work of Davison and Mandel \cite{davison2021strong} in order to prove the strong positivity conjecture for quantum cluster algebras.

%\begin{thm}[{\cite[Prop. 3.2.6, 3.3.2]{KS}}] Let $\frakD_{\init}$ be a finite scattering diagram in $M_\R$ over $\calT$ whose only walls are incoming.  Then there is a unique (up to equivalence) consistent scattering diagram $\frakD$, denoted by $\Scat(\frakD_{\init})$, containing $\frakD_{\init}$ such that $\frakD \cut \frakD_{\init}$ consists only of outgoing walls. \end{thm}

\begin{figure}
\begin{tikzpicture}[scale=9]
\def\v{0.15cm}
\draw[line width=0.1pt] (-.75,0)--(.4,0);
\draw[line width=0.1pt] (0,-.75)--(0,.4);
\draw[line width=0.1pt] (0,0)--({atan(1)}:{-.8 });
\foreach \x in {2,...,2}
    {
    \draw[line width=0.1pt] (0,0)--({atan(\x/(\x-1))}:{-.8 });
    }
\foreach \x in {2,...,3}
    {
    \draw[line width=0.1pt] (0,0)--({atan((\x-1)/\x)}:{-.8 });
    }
\draw[line width=0.2pt, color=blue] (0.25,0.375) node[circle,fill=black,minimum size=\v,inner sep=0pt]{}--(0,0.25) node[circle,fill=blue,minimum size=\v,inner sep=0pt]{}--(-1/8,-1/16) node[circle,fill=blue,minimum size=\v,inner sep=0pt]{}--(-1/6,-1/9) node[circle,fill=blue,minimum size=\v,inner sep=0pt]{}--(-.61,-0.5175);

\draw[line width=0.1pt, color=red] (0.25,0.375) --(0.1, 0.0) node[circle,fill=red,minimum size=\v,inner sep=0pt]{}--(-0.045454545454545456, -0.09090909090909091) node[circle,fill=red,minimum size=\v,inner sep=0pt]{}--(-0.5, -0.5) node[circle,fill=red,minimum size=\v,inner sep=0pt]{}--(-0.5685, -0.5628);

\draw (-.85,-.5) node[anchor=west,color=blue]  {\small$(-12,-11)$};
\draw (-.37,-.07) node[anchor=west,color=blue]  {\small$(-6,-7)$};
\draw (-.25,0.12) node[anchor=west,color=blue]  {\small$(-2,-5)$};
\draw (0,0.37) node[anchor=west,color=blue]  {\small$(-2,-1)$};

\draw (-.8,-.57) node[anchor=west,color=red]  {\small$(-12,-11)$};
\draw (-.44,-.45) node[anchor=west,color=red]  {\small$(-10,-9)$};
\draw (0.01,-.08) node[anchor=west,color=red]  {\small$(-8,-5)$};
\draw (0.17,0.18) node[anchor=west,color=red]  {\small$(-2,-5)$};

\end{tikzpicture}
\caption{Two broken lines on $\frakD_2$ with initial slope $(-12,-11)$ are shown above. Though there are infinitely many walls in $\frakD_2$, we do not depict those with slope in the range $(2/3,1) \cup (1,2)$.  The red line has positive angular momentum, while the blue broken line has negative angular momentum and is the same as in \autoref{fig: broken line q weight}.  The terminal point $(q_1,q_2) = (0.25,0.375)$ is shown in black in the first quadrant.  The points at which the broken lines bend are shown as nodes of the corresponding color, and the exponent for each domain of linearity is shown nearby in the corresponding color.}
\label{fig: broken line angular momentum}
\end{figure}
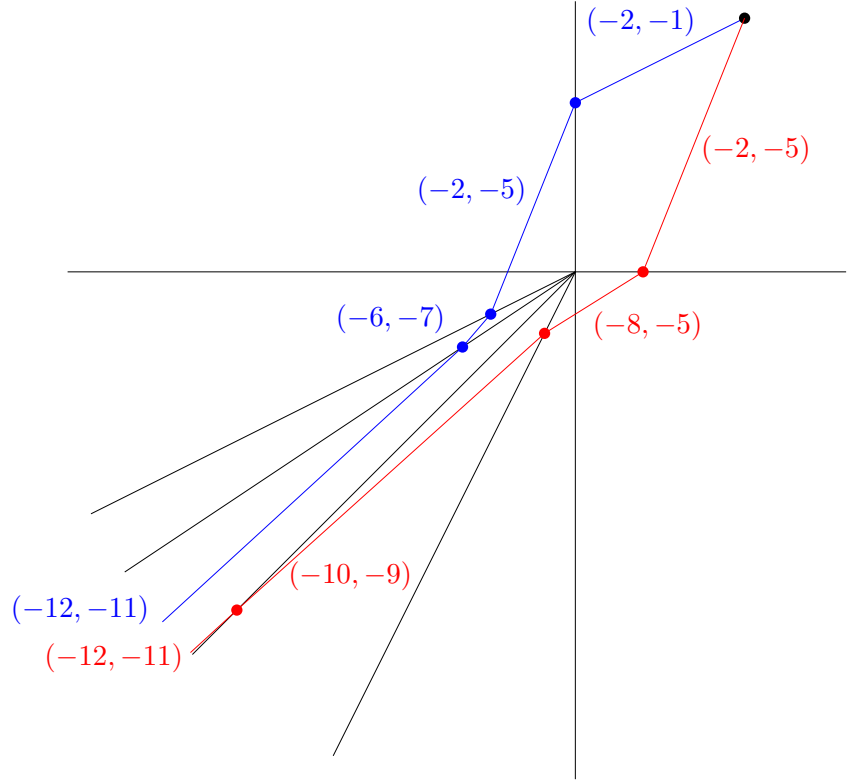

\subsection{Angular Momentum}\label{subsec: ang momentum}
We now discuss an invariant of broken lines in rank $2$ scattering diagrams that will help us distinguish which broken lines cross the Badlands.  Suppose that we are considering the set of compatible pairs corresponding to a monomial $X_1^{p_1}X_2^{p_2}$ in the cluster variable $X_n$.  

\begin{defn}
Fix a point $(q_1,q_2)$ on the linear portion with slope $(m_1,m_2)$ of a broken line terminating in the first quadrant.  The \emph{angular momentum} of the broken line at that point is the quantity $q_2m_1 - q_1m_2$.
\end{defn}

\begin{lem}[\text{\cite[Lemma 5.3]{CGM}}]
The angular momentum is constant on a broken line.
\end{lem}

We will be primarily interested in whether the angular momentum is positive or negative.  The terminal exponent of a broken line can sometimes determine the sign of its angular momentum, and otherwise its sign also depends on the choice of terminal point.

\begin{lem}\label{lem: angular momentum sign}
Let $\gamma$ be a broken line in $\frakD_r$ terminating at $(q_1,q_2)$ in the first quadrant that contributes to the monomial $X_1^{-c_{n-1}+rb}X_2^{-c_{n-2}+ra}$ in the quantum cluster variable $X_n$.  Then
\begin{enumerate}
    \item\label{case: ams 1} if $rb \geq c_{n-1}$, the angular momentum of $\gamma$ is positive,
    \item\label{case: ams 2} if $ra \geq c_{n-2}$, the angular momentum of $\gamma$ is negative,
    \item\label{case: ams 3} otherwise, the angular momentum of $\gamma$ can be either positive, zero, or negative depending on the choice of $(q_1,q_2)$.
\end{enumerate}
\end{lem}
\begin{proof}
Since $q_1,q_2 > 0$, this implies that in cases (\ref{case: ams 1}) and (\ref{case: ams 2}), the angular momentum is never $0$.  The claims about the corresponding signs in cases (\ref{case: ams 1}) and (\ref{case: ams 2}) follow from the signs of the terms in the formula for angular momentum.  Since there are no walls in the interior of the first quadrant, it follows that $m_1 = -c_{n-1} + rb$ and $m_2 = -c_{n-2} + ra$ cannot both be positive.  Thus, case \ref{case: ams 3} handles when both $m_1$ and $m_2$ are negative, in which case choosing $q_1 \gg q_2$ or $q_2 \gg q_1$ will change the sign of the angular momentum.
\end{proof}

\begin{lem}[\text{\cite[Lemma 5.4]{CGM}}]\label{lem: broken line slope}
Let $\gamma$ be a broken line in $\frakD_r$ terminating in the first quadrant.  If $\gamma$ has positive (\resp negative) angular momentum, then the slope of the linear portions of $\gamma$ decreases (\resp increases) at each bend, except possibly at the boundary of the first quadrant.
\end{lem}

\begin{rem}
It follows that a broken line with initial exponent $(m_1,m_2)$ will only bend at walls of slope less than (\resp greater than) $\frac{m_2}{m_1}$ if the broken line has negative (\resp positive) angular momentum.  Thus, $\gamma$ is a broken line with negative angular momentum associated to a cluster monomial in $\frakD_r$, then $\gamma$ only bends at walls above the Badlands. 
\end{rem}

\begin{defn}
Suppose we have integers $\ell_i \geq 0$ and $m_i \geq 1$ for each $i \in \{1,\dots,k\}$, where the $\ell_i$ are decreasing. Let $\ell$ be the sequence $[(\ell_1,m_1),(\ell_2,m_2),\dots,(\ell_k,m_k)]$.  Let $\calL^-_{r,n}[\ell] = \calL^-_{r,n}[(\ell_1,m_1),(\ell_2,m_2),\dots,(\ell_k,m_k)]$ be the broken line in $\frakD_r$ with initial slope $(-c_{n-1},-c_{n-2})$ that bends at the wall of slope $\frac{c_{\ell_{i}}}{c_{\ell_{i}+1}}$ with multiplicity $m_i$ and terminates at $(q_1,q_2)$.   By convention, the wall of slope $\frac{-1}{0}$ or $\frac{1}{0}$ is the $y$-axis. 

Similarly, let $\calL^+_{r,n}[\ell] = \calL^+_{r,n}[(\ell_1,m_1),(\ell_2,m_2),\dots,(\ell_k,m_k)]$ be the broken line in $\frakD_r$ with initial slope $(-c_{n-1},-c_{n-2})$ that bends at the wall of slope $\frac{c_{\ell_{i} + 1}}{c_{\ell_{i}}}$ with multiplicity $m_i$ and terminates at $(q_1,q_2)$.  Note that $\calL^-_{r,n}[\ell]$ has negative angular momentum while $\calL^+_{r,n}[\ell]$ has positive angular momentum. When $r = 2$, we additionally consider broken lines of the form $\calL^+_{r,n}[(\infty, m_\infty), (\ell_1,m_1),(\ell_2,m_2),\dots,(\ell_k,m_k)]$, where the broken line bends over the wall of slope $1 = \lim_{\ell \to \infty} \frac{c_{\ell + 1}}{c_{\ell}}$ with multiplicity $m_\infty$.
\end{defn}

\begin{exmp}\label{exmp: neg broken line bends}
The blue broken line in \autoref{fig: broken line angular momentum} is $\calL^-_{2,14}[(3,1),(2,1),(0,2)]$, and the red broken line is $\calL^+_{2,14}[(\infty,1),(2,1),(0,3)]$.
\end{exmp}

\section{Weights of Broken Lines}\label{sec: weights of bls}
\subsection{The Cluster Variable Case}
For any integer $r$, and for any real numbers $\alpha,\beta$, let $f^r_{\alpha,\beta}$ denote the sequence where $f^r_{\alpha,\beta}(0) = \alpha$, $f^r_{\alpha,\beta}(1) = \beta$, and $f^r_{\alpha,\beta}(k+1) = rf^r_{\alpha,\beta}(k) - f^r_{\alpha,\beta}(k-1)$ for $k \geq 1$.  

\begin{lem}\label{lem: sequence recursion}
For any $k \in \Z_{\geq 0}$ and real numbers $\alpha,\beta$, we have
$$c_{k+1}\beta - c_k\alpha = f^r_{\alpha,\beta}(k)$$
\end{lem}
\begin{proof}
Let $g(k) = c_{k+1}\beta - c_k\alpha$.  We first see that 
\begin{align*}g(k) &= (rc_k - c_{k-1})\beta - (rc_{k-1}-c_{k-2})\alpha\\
&= r(c_k\beta - c_{k-1}\alpha) - (c_{k-1}\beta - c_{k-2}\alpha)\\
&= rg(k-1)-g(k-2)\,.
\end{align*}
Moreover, we have $g(0) = \alpha$ and $g(1) = \beta$, so we can conclude $g(k) = f^r_{\alpha,\beta}(k)$.
\end{proof}

\begin{cor}\label{cor: cn identities}
For any $n \geq \ell \in \Z_{>0}$, we have
$$c_nc_{\ell+1} - c_{n+1}c_\ell  = c_{n-\ell + 1}$$
and 
$$c_{n+1}c_{\ell+1} - c_nc_\ell = c_{n + \ell}\,.$$
\end{cor}
\begin{proof}
For the first statement, apply \autoref{lem: sequence recursion} with $\alpha = c_{h + 1}$ and $\beta = c_h$.  It is then straightforward to check that $f^r_{c_{h+1},c_h}(\ell) = c_{h-\ell + 1}$.

For the second statement, apply \autoref{lem: sequence recursion} with $\alpha = c_h$ and $\beta = c_{h+1}$.  It is then straightforward to check that $f^r_{c_{h},c_{h+1}}(\ell) = c_{h + \ell}$.
\end{proof}

\begin{prop}\label{prop: partition formula negative am}
The quantum weight of the broken line $\calL^-_{r,n+2}[(\ell_1,m_1),\dots,(\ell_k,m_k)]$ is
$$\prod_{i = 1}^k \binom{c_{n-\ell_i + 1} - r\sum_{j < i} m_{\ell_j}c_{\ell_j - \ell_i+1}}{m_i}_{q^{2r}}\,.$$
\end{prop}
\begin{proof}
We prove this by induction on the number of crossings.  Note that the broken line $\calL^-_{r,n}[(\ell_1,m_1),\dots,(\ell_{k-1},m_{k-1})]$ has terminal exponent $$\left(-c_{n+1} + \sum_{j=1}^{k-1} m_j c_{\ell_j + 1} ,-c_n + \sum_{j=1}^{k-1} m_j c_{\ell_j}\right)\,.$$ 
Using \autoref{cor: cn identities}, we can calculate that the dot product of this terminal exponent and the primitive normal vector to the wall of slope $\frac{c_{\ell_k}}{c_{\ell_k+1}}$ is $d \colonequals c_{n-\ell_k + 1} - r\sum_{j < k} m_{\ell_j}c_{\ell_j - \ell_k+1}$.  The desired formula is then readily calculated by multiplying the quantum weight of $\calL^-_{r,n+2}[(\ell_1,m_1),\dots,(\ell_{k-1},m_{k-1})]$ by the binomial coefficient $\binom{d}{m_k}_{q^{2r}}$.
\end{proof}
\begin{rem}
In the above statement, $m_1$ is the multiplicity of the bending of $\gamma$ at the $y$-axis.  \end{rem}

\begin{prop}\label{prop: partition formula positive am}
The quantum weight of the broken line $\calL^+_{r,n+2}[(\ell_1,m_1),\dots,(\ell_k,m_k)]$ is
$$\prod_{i = 1}^k \binom{c_{n+\ell_i} - r\sum_{j < i} m_{\ell_j}c_{\ell_j - \ell_i}}{m_i}_{q^{2r}}\,.$$
\end{prop}
\begin{proof}
This follows from a similar argument to that for \autoref{prop: partition formula negative am}, except using the second equality in \autoref{cor: cn identities} rather than the first.
\end{proof}

\begin{defn}
Let $\BL(\ell,h,a,b)$ be the collection of broken lines with initial exponent $(-\ell,-h)$ and terminal exponent $(-\ell + rb, -h + ra)$.   Let $\BL_-(\ell,h,a,b)$ (\resp $\BL_+(\ell,h,a,b)$) be the set of broken lines in $\BL(\ell,h,a,b)$ with negative (\resp positive) angular momentum.

Given a set of broken lines $B$, let $|B|_q$ denote the sum of the weights of the broken lines in $B$.  We define $|B| \in \Z$ to be the value of $|B|_q$ under the substitution $q = 1$.
\end{defn}

\begin{lem}\label{lem: broken line recursion}
If $rb \leq c_n$, we have that
$$|\BL_-(c_{n+1},c_n,a,b)|_q= \sum_{t=0}^{a} \binom{c_{n+1} - rb}{t}_{q^{2r}}|\BL_-(c_n,c_{n-1},r(a-t)-b,a - t)|_q$$
\end{lem}
\begin{proof}
Fix a broken line $\gamma = \calL^-_{r,n+2}[(\ell_1,m_1),\dots,(\ell_k,m_k)]$ in $\BL_-(c_{n+1},c_n,a,b)$.  Let $t$ denote the multiplicity of the crossing over the $y$-axis, i.e., $t = m_k$ if $\ell_k = 0$ and $t = 0$ otherwise.  If $t = 0$, let $\gamma' = \calL^-_{r,n+1}[(\ell_1 - 1,m_1),\dots,(\ell_{k-1} - 1,m_{k-1})]$, and otherwise let $\gamma' = \calL^-_{r,n+1}[(\ell_1 - 1,m_1),\dots,(\ell_{k} - 1,m_{k})]$.  Note that we have $\gamma' \in \BL_-(c_n,c_{n-1},r(a-t)-b,a - t)$.  By \autoref{prop: partition formula negative am}, we have
$|\gamma|_q = \binom{c_{n+1} - rb}{t}_q |\gamma'|_q$. 
Note that the map from $\gamma$ to $\gamma'$ is bijective (considered as a map on unweighted broken lines).  Therefore, summing over all choices of $\gamma$ in $\BL_-(c_{n+1},c_n,a,b)$ yields the desired equality.
\end{proof}

\subsection{The Kronecker Case}
We first handle broken lines of positive angular momentum and then broken lines of initial exponent $(-h,-h)$.  The latter are the broken lines corresponding to the theta basis elements of $\calA_q(2,2)$ that are not cluster monomials.  

\begin{prop}\label{prop: partition formula kron positive am}
In the Kronecker scattering diagram $\frakD_2$, the quantum weight of the broken line $\calL^+_{r,n+2}[(\infty,m_\infty), (\ell_1,m_1),\dots,(\ell_k,m_k)]$ is
$$[m_\infty + 1]_{q^{4}}\prod_{i = 1}^k \binom{(n+ \ell_i - 2m_\infty) - 2\sum_{j < i} m_{\ell_j}(\ell_j - \ell_i)}{m_i}_{q^{4}}\,.$$
\end{prop}

\begin{lem}\label{lem: bl recursion pos}
In the Kronecker scattering diagram $\frakD_2$, we have that
$$|\BL_+(h+1,h,a,b)|_q = \sum_{t=0}^{b} \binom{h - 2a}{t}_{q^{4}}|\BL_+(h,h-1,b - t,2(b-t)-a)|_q$$
whenever $2a \leq h+1$.
\end{lem}
\begin{proof}
Fix a broken line $\gamma = \calL^+_{r,h+3}[(\infty,m_\infty), (\ell_1,m_1),\dots,(\ell_k,m_k)]$ in $\BL_-(h+1,h,a,b)$.  Let $t$ denote the multiplicity of the crossing over the $x$-axis, i.e., $t = m_k$ if $\ell_k = 0$ and $t = 0$ otherwise.  If $t = 0$, let $\gamma' = \calL^+_{r,h+2}[(\infty,m_\infty),(\ell_1 - 1,m_1),\dots,(\ell_{k-1} - 1,m_{k-1})]$, and otherwise let $\gamma' = \calL^+_{r,h+2}[(\infty,m_\infty),(\ell_1 - 1,m_1),\dots,(\ell_{k} - 1,m_{k})]$.  Note that we have $\gamma' \in \BL_-(h,h-1,b - t,2(b-t)-a)$.  By \autoref{prop: partition formula kron positive am}, we have
$|\gamma|_q = \binom{c_h - 2b}{t}_q |\gamma'|_q$. 
Note that the map from $\gamma$ to $\gamma'$ is bijective (considered as a map on unweighted broken lines).  Therefore, summing over all choices of $\gamma$ in $\BL_-(h+1,h,a,b)$ yields the desired equality.
\end{proof}

\begin{prop}\label{prop: partition formula kron negative am}
Suppose $r = 2$ and $2b \leq h$.  The quantum weight of the broken line in $\BL_-(h,h,a,b)$ crossing at the wall $\frakd_{\ell_i/(\ell_i + 1)}$ with multiplicity $m_i$ for $\ell_1 > \cdots > \ell_k \geq -1$ is
$$\prod_{i = 1}^k \binom{h - \sum_{j=1}^{i-1} 2m_j(\ell_i - \ell_j)}{m_i}_{q^{4}}\,. $$
\end{prop}
\begin{proof}
    Denote this broken line by $\gamma$.  Let $p_i$ be the exponent of $\gamma$ after bending at the wall $\frakd_{\ell_i/(\ell_i+1)}$, and let $p_0 = (-h,-h)$.  It is straightforward to prove by induction that we have 
    $$p_i = \left( -h + 2\sum_{j=-1}^i m_j \ell_j, -h + 2\sum_{j=-1}^i m_j(\ell_j + 1)\right)\,$$
    Thus, the dot product of $p_{i-1}$ with the primitive normal vector to the wall $\frakd_{\ell_i/(\ell_i+1)}$ is 
    $$p_{i-1} \cdot (-\ell_{i}, \ell_{i}+1) =  h - \sum_{j=1}^{i-1} 2m_j(\ell_j - \ell_i)\,.$$ 
    The weight of the broken line is then the product of quantum binomial coefficients $\binom{p_{i-1} \cdot (-\ell_{i}, \ell_{i}+1)}{m_i}_{q^4}$.
\end{proof}

\begin{lem}\label{lem: broken line recursion kronecker greedy}
If $r = 2$ and $2b \leq h$, we have that
$$|\BL_-(h,h,a,b)|_q = \sum_{t=0}^{a} \binom{h - 2b}{t}_{q^{4}}|\BL_-(h,h,2(a-t)-b,a-t)|_q\,.$$
\end{lem}
\begin{proof}
The proof is essentially the same as that of \autoref{lem: bl recursion pos}, but adapted for the initial exponent $(-h,-h)$.  Consider the broken line $\gamma \in \BL_-(h,h,a,b)$ that bends over the wall $\frac{\ell-1}{\ell}$ with multiplicity $m_\ell$ for $\ell \geq 0$.  Let $m_{0} = t$.  We then map $\gamma$ to $\gamma' \in |\BL_-(h,h,2(a-t)-b,a-t)|$ that bends over the wall $\frac{\ell-2}{\ell-1}$ with multiplicity $m_\ell$ for $\ell \geq 0$.  By \autoref{prop: partition formula kron negative am}, we can see that $w_q(\gamma) = \binom{h - 2b}{t}_{q^4} w_q(\gamma')$.  The desired equality follows from summing over all possible values of $m_{0} = t$.
\end{proof}

\section{\texorpdfstring{A Bijection for $r$-Kronecker Cluster Monomials}{A Bijection for r-Kronecker Cluster Monomials}}\label{sec: bijection for r-kronecker}

\subsection{Construction of the Bijection}
We now construct a $q$-weighted bijection between positive compatible pairs and broken lines of negative angular momentum associated to quantum cluster monomials in the $r$-Kronecker cluster algebra.  That is, for $\ell \geq rb$,  we construct a map $\phi: \CP(\ell,h,a,b) \to \BL_-(\ell,h,a,b)$ such that, for a broken line $\gamma \in \BL_-(\ell,h,a,b)$, we have
$$\sum_{(S_1,S_2) \in \phi^{-1}(\gamma)} w_q(S_1,S_2) = |\gamma|_q\,.$$

\begin{defn}
A vertical edge $\nu$ is \emph{overshadowing} if the horizontal edge shadow-paired with $\nu$ is in $S_1$.  A vertical edge $\nu$ is \emph{overflowing} if the horizontal edge cascade-paired with $\nu$ is in $S_1$.  
\end{defn}

We say that a vertical edge of $\calP$ is \emph{protruding} if it has fewer than $r$ horizontal edges to its immediate left.  Note that, by the compatibility condition, if a vertical edge is overshadowing or overflowing, then it is protruding.  We denote the protruding edges of $\calP$ by $\nu_{\pro,1},\dots,\nu_{\pro,rh-\ell}$, ordered from bottom to top.

Suppose $(r-1)h \leq \ell \leq rh$.  Let $\ell' = h$ and $h' = rh - \ell$, and consider the paths $\calP = \calP(\ell,h)$ and $\calP' = \calP(\ell',h')$.  Note that the path $\calP'$ is obtained from $\calP$ by replacing the sequence of steps $E^{r-1}N$ with $EN$ and $E^rN$ with $E$.

\begin{defn}\label{defn: tilde theta}
We define a map $\ttheta:\CP(\calP) \to \CP(\calP)$ taking $(S_1,S_2)$ to $(\widetilde S_1,\widetilde S_2)$, where
\begin{itemize}
\item $\widetilde S_1 = \{\eta_i \in \calP': \nu_i \in \calP \cut S_2 \}$, and
\item $\widetilde S_2 = \{\nu_j \in \calP' : \nu_{\pro,j} \text{ is overflowing}\}$
\end{itemize}    
\end{defn}
\begin{rem}
It is not clear from the definition that the resulting set of edges in the image of $\ttheta$ is a compatible pair.  This fact is proven in \autoref{cor: ttheta compatibility}.  Moreover, when restricted to the set $\CP_{\cas}(\ell,h,a,b)$, the map $\ttheta$ yields a compatible pair in $\CP(\ell',h',h-b,a)$.
\end{rem}

Let $(S_1^{(i)},S_2^{(i)}) = \underbrace{\ttheta \circ \ttheta \circ \cdots \circ \ttheta}_{i}(S_1,S_2)$ for $i \geq 0$.  Let $b_i = |S_2^{(i)}|$.  Note that the sequence $b_i$ is weakly decreasing, and we have $b_i = 0$ for sufficiently large $i$. 

\begin{defn}\label{defn: phi}
Let $\calP(\ell,h)$ be a Dyck path corresponding to a cluster monomial.  We define $\phi:\CP(\ell,h,a,b)\to \BL_{-}(\ell,h,a,b)$ to be the map taking a compatible pair $(S_1,S_2)$ to the broken line that crosses the wall of slope $\frac{c_i}{c_{i-1}}$ with multiplicity $b_{i-1}-(rb_i-b_{i+1})$.     
\end{defn}

A potential direction for further work would be extending this bijection to all broken lines that arise in the theta bases of rank-$2$ (quantum) cluster algebras.  This problem was posed in the classical case in \cite[Remark 3.6]{CGM}, as this would give a combinatorial proof that the greedy and theta bases coincide.  In the quantum setting, it is not yet known if these bases coincide, though this was suggested to be true in \cite[Section 1.4]{davison2021strong}. For the quantum Kronecker cluster algebra, we show at the end of \autoref{sec: bijection for kronecker} that the quantum greedy and theta bases indeed coincide.

\subsection{Proof of Compatibility}
We begin by recalling a map $\theta$ that takes a compatible pair to another compatible pair on a different Dyck path.  This map originally appeared in the work of Lee--Li--Zelevinsky and was used in their proof of \autoref{thm: LLZ expansion}.  

\begin{defn}[{\cite[Lemma 3.5]{LLZ}}]\label{def: theta equivalent def}
If $(r-1)h \leq \ell \leq rh$, then the map $\theta:\CP_{\cas}(\ell,h,a,b) \to \CP(\ell',h',h-b,a)$ taking $(S_1,S_2)$ to $(S_1',S_2')$ is defined as follows:
\begin{itemize}
\item $S_1' = \{\eta_i \in \calP' : \nu_{h + 1-i} \text{ is overshadowing}\}$, and
\item $S_2' = \{\nu_j \in \calP' : \nu_j \in \calP \cut S_2\}$.
\end{itemize}
\end{defn}

\begin{rem}
Our definition of the Lee-Li-Zelevinsky map $\theta$ differs from their original formulation.  Moreover, we apply $\theta$ to compatible pairs (rather than sets of horizontal edges).  It is straightforward to show that the two definitions are equivalent (up to conjugating the underlying Dyck path and including the vertical edges). 

%Essentially, their map swaps an edge at height i shadowed by the j-th vertical edge with an edge at height h-j shadowed by the i-th vertical edge.  In our formulation, height corresponds to which vertical edge is shadowing, so this swap makes sense
\end{rem}

We now prove that any set of edges $(S_1,S_2)$ in the image of $\widetilde \theta$ is indeed a compatible pair.  To do this, we utilize that the map $\theta$ has this property.

\begin{thm}\label{lem: theta composition}
We have $\theta(S_1,S_2) = (\ttheta \circ \lambda)(S_1,S_2)$ for any $(S_1,S_2) \in \CP_{\sh}(\ell,h,a,b)$.  
\end{thm}
\begin{proof}
    Note that $\lambda$ does not affect $S_2$ and both $\theta$ and $\ttheta$ replace $S_2$ with its complement.  The map $\theta$ depends on which edges of $S_2$ are overshadowing in $(S_1,S_2)$, while the map $\ttheta$ depends analogously on which edges of $S_2$ are overflowing in $\lambda(S_1,S_2)$.  By \autoref{lem: lambda shadow exchange}, the overshadowing edges of $(S_1,S_2)$ coincide with the overflowing edges of $\lambda(S_1,S_2)$.  Therefore, we can conclude that the maps are the same.
\end{proof}

\begin{cor}\label{cor: ttheta compatibility}
For compatible pairs on Dyck paths corresponding to cluster monomials, the map $\ttheta$ preserves compatibility.
\end{cor}
\begin{proof}
By \autoref{lem: lambda inverse compatibility}, the map $\lambda^{-1}$ preserves compatibility for the cluster monomials (which includes the setting in which $\ttheta$ is defined).  By \autoref{lem: theta composition}, we have $\ttheta = \theta \circ \lambda^{-1}$.  Since the map $\theta$ preserves compatibility (see \cite[Lemma 3.5]{LLZ}), we can conclude that $\ttheta$ does as well.
\end{proof}

\subsection{Preservation of Quantum Weight}\label{subsec: preserving quantum weight}

We now show that, on a path $\calP(\ell,h)$ corresponding to a cluster monomial, the map $\ttheta$ preserves the quantum weights of compatible pairs.  Thus, the map $\phi: \CP(\ell,h,a,b) \to \BL_-(\ell,h,a,b)$ is a $q$-weighted bijection (see \autoref{thm: phi q weighted bijection}).  Throughout this section, let $(\widetilde S_1,\widetilde S_2) = \ttheta(S_1,S_2)$.  

\begin{prop}\label{prop: gen hor swap}
    Let $\omega \in \{h,v,H,V\}^*$ be the word associated to a compatible pair in $\CP(\ell,h,a,b)$ where $\calP(\ell,h)$ corresponds to a cluster monomial.  If we swap an instance $h$ with an $H$ to its right such that the segment $\sigma$ between them satisfies $|\sigma_h| + |\sigma_H| = r|\sigma_V|$, then the quantum weight increases by $2r$.
\end{prop}
\begin{proof}
    Note that the above expression does not take into account $|\sigma_v|$. Moreover, the quantum weight contributed from portions of $\omega$ outside the segment $h\sigma H$ (or $H\sigma h$) are unaffected.  We can then calculate that 
    \begin{align*}
        w_q(H\sigma h) &= w_q(Hh) + w_q(|\sigma_H|HH + |\sigma_V|HV +  |\sigma_h|Hh + |\sigma_v|Hv)\\ 
        &\hspace{1.8cm} + w_q(|\sigma_H|Hh + |\sigma_V|Vh + |\sigma_h|hh + |\sigma_v|vh)\\
        &= r + |\sigma_V|(1-r^2) + |\sigma_h|r + |\sigma_v| + |\sigma_H|r - |\sigma_V| - |\sigma_v|\\
        &= r + |\sigma_V|(1-r^2) +  (r|\sigma_V| - |\sigma_H|)r + |\sigma_H|r - |\sigma_V|\\
        &= r
    \end{align*}
    and 
    \begin{align*}
        w_q(h\sigma H) &= w_q(hH) + w_q(|\sigma_H|hH + |\sigma_V|hV +  |\sigma_h|hh + |\sigma_v|hv)\\ 
        &\hspace{1.8cm} + w_q(|\sigma_H|HH + |\sigma_V|VH + |\sigma_h|hH + |\sigma_v|vH)\\
        &= -r -|\sigma_H|r + |\sigma_V| + |\sigma_v| + |\sigma_V|(r^2 -1) - |\sigma_h|r - |\sigma_v|\\
        &= -r -(r|\sigma_V| - |\sigma_h|)r + |\sigma_V| + |\sigma_V|(r^2 - 1) - |\sigma_h|r\\
        &= -r\,.
    \end{align*}
The resulting change in the quantum weight is then $w_q(H\sigma h - h\sigma H) = 2r$, as desired.
\end{proof}

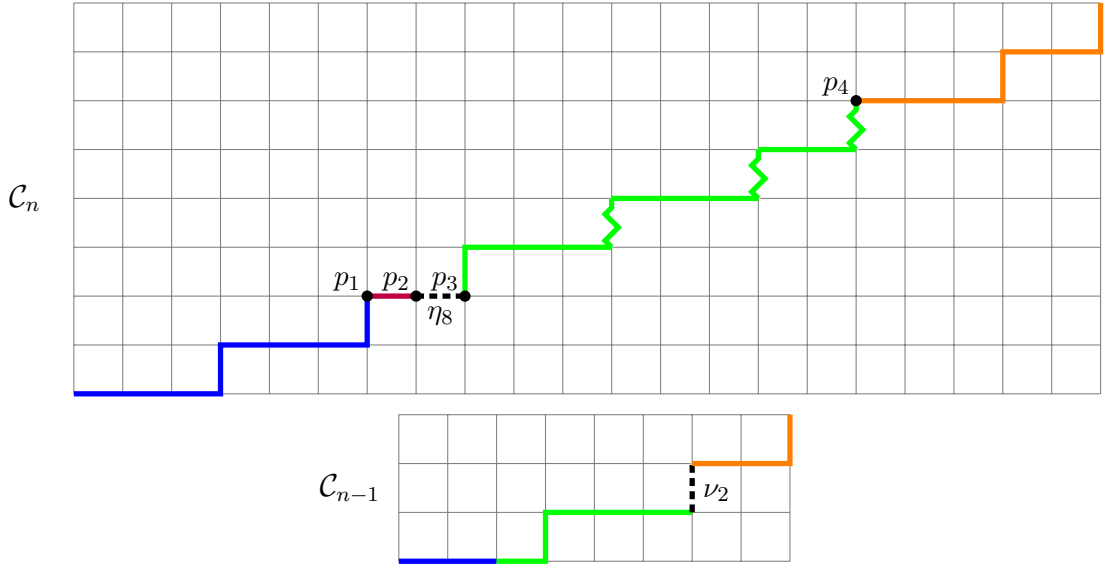
\begin{figure}
\[
\begin{tikzpicture}[scale=.65]
\def\v{0.15cm}
\draw[step=1,color=gray] (0,0) grid (21,8);
\draw[color = white, line width=2pt] (16,5)--(16,6);
\draw[color = white, line width=2pt] (14,4)--(14,5);
\draw[color = white, line width=2pt] (11,3)--(11,4);
\draw[color = white, line width=2pt] (7,2)--(8,2);
\draw[decorate,decoration=zigzag, color = green, line width=2pt] (16,5)--(16,6);
\draw[decorate,decoration=zigzag, color = green, line width=2pt] (14,4)--(14,5);
\draw[decorate,decoration=zigzag, color = green, line width=2pt] (11,3)--(11,4);
\draw[line width=2,color=blue] (0,0)--(3,0)--(3,1)--(6,1)--(6,2);
\draw[line width=2,color=purple] (6,2)  node[circle,fill=black,minimum size=\v,inner sep=0pt]{}--(7,2);
\draw[line width=2,color=black, densely dashed] (7,2) node[circle,fill=black,minimum size=\v,inner sep=0pt]{}--(8,2);
\draw[line width=2,color=green] (8,2) node[circle,fill=black,minimum size=\v,inner sep=0pt]{}--(8,3)--(11,3);
\draw[line width=2,color=green] (11,4)--(14,4);
\draw[line width=2,color=green] (14,5)--(16,5);
\draw[line width=2,color=orange] (16,6) node[circle,fill=black,minimum size=\v,inner sep=0pt]{}--(19,6)--(19,7)--(21,7)--(21,8);
\draw (5.6,2.7) node[anchor=north,color=black]  {\small$p_1$};
\draw (6.6,2.7) node[anchor=north,color=black]  {\small$p_2$};
\draw (7.6,2.7) node[anchor=north,color=black]  {\small$p_3$};
\draw (15.6,6.7) node[anchor=north,color=black]  {\small$p_4$};
\draw (7.5,2) node[anchor=north,color=black]  {\small$\eta_8$};
\draw (-1,4.5) node[anchor=north,color=black]  {$\calC_n$};
\end{tikzpicture}
\]
\[
\begin{tikzpicture}[scale=.65]
\def\v{0.15cm}
\draw[step=1,color=gray] (0,0) grid (8,3);
\draw[color = white, line width=2pt] (6,1)--(6,2);
\draw[line width=2,color=blue] (0,0)--(2,0);
\draw[line width=2,color=green] (2,0)--(3,0)--(3,1)--(6,1);
\draw[line width=2,color=black, densely dashed] (6,1)--(6,2);
\draw[line width=2,color=orange] (6,2)--(8,2)--(8,3);
\draw (-1,2) node[anchor=north,color=black]  {$\calC_{n-1}$};
\draw (6.5,1.8) node[anchor=north,color=black]  {\small$\nu_2$};
\end{tikzpicture}
\]
\caption{An illustration of the construction in \autoref{lem: add shadowed}.}
\end{figure}

Let $L$ be a set of letters.  Let $|W|_L$ denote the number of entries of $W$ that are in $L$ (with multiplicity).  For integers $s_i$ and words $W_i$, let $\left|\sum_{1 \leq i \leq k} s_iW_i\right|_L = \sum_{1 \leq i \leq k} s_i|W_i|_L$.

\begin{lem}\label{lem: add shadowed}
    Fix a positive compatible pair $(S_1,S_2)$ on $\calC_n$. Suppose that the horizontal edge $\eta_i \notin S_1$ is cascade-paired with $\nu_j \in S_2$ and that $\eta_i$ is to the left of $\nu_j$. Then 
    $$w_q(S_1 \cup \{\eta_i\},S_2) - w_q(S_1,S_2) = w_q(\widetilde S_1, \widetilde S_2 \cup \{\nu_{j'}\}) - w_q(\widetilde S_1,\widetilde S_2)\,.$$
\end{lem}
\begin{proof}
We begin by decomposing $\calC_n$ into several paths.
Let $p_1$ be the leftmost point along $\calC_n$ at the same height at $\eta_i$.  Let $p_2$ be the left endpoint of $\eta_i$ and let $p_3$ be the top endpoint of $\nu_j$.    Let $X$ be subpath of $\calC_n$ consisting of all edges below $p_1$, and let $Y$ be the path from $p_1$ to $p_2$. Let $S$ be the path from $p_2$ to $p_3$, and let $Z$ be the subpath of $\calC_n$ consisting of all edges to the right of $p_3$.

We similarly decompose $\calC_{n-1}$ into several paths.  Let $\widetilde X$ be the image of $X$, $\widetilde S$ be the image of $Y \cup \{\eta_i\} \cup S$ with $\nu_{j'}$ removed, and $\widetilde Z$ be the image of $Z$.

Expanding out the pairs of edges that involve $\eta_i$ in $\calC_n$, we have
\begin{align*}
w_q(S_1 \cup \{\eta_i\},S_2) - w_q(S_1,S_2) &= w_q(XYHSZ - XYhSZ) \\
&= r|S + Z - X - Y|_{h,H} + r^2|X+Y-S-Z|_V 
\end{align*}
and 
$$w_q(\widetilde S_1, \widetilde S_2 \cup \{\eta_j\}) - w_q(\widetilde S_1,\widetilde S_2) = r|\widetilde X + \widetilde S - \widetilde Z|_{v,V} + r^2|\widetilde Z - \widetilde X - \widetilde S|_H\,.$$
Thus, we want to show that the quantity 
\begin{align*}
\alpha &\colonequals r\left(|S + Z - X - Y|_{h,H} + |\widetilde Z - \widetilde X - \widetilde S|_{v,V}\right)\\
&\qquad+ r^2\left(|X+Y-S-Z|_{V} + |\widetilde X + \widetilde S - \widetilde Z|_H\right)
\end{align*}
vanishes.  
By the definition of the map $\rho$, we have 
$$r|X|_{v,V} - |X|_{h,H} = |\widetilde X|_{v,V} \text{ and } r|Z|_{v,V} - |Z|_{h,H} = |\widetilde Z|_{v,V}\,.$$
Hence 
$$r^2|X-Z|_V + r|Z-X|_{h,H} + r|\widetilde Z-\widetilde X|_{v,V} = r^2|Z-X|_v\,.$$
Plugging this into our expression for $\alpha$, we find 
\begin{align*}
\alpha &= r\left(|S - Y|_{h,H} - | \widetilde S|_{v,V}\right) + r^2\left(|Y-S|_{V} + |\widetilde X + \widetilde S - \widetilde Z|_H + |Z-X|_v \right)\,.
\end{align*}

By the definition of the map $\rho$, we have $|W|_V = |\widetilde W|_{h}$ and $|W|_v = |\widetilde W|_{H}$ for $W = X,S,Z$.  We also must have $|Y|_{v,V} = 0$.  Thus, we can further simplify $\alpha$ to 
\begin{align*}
\alpha &= r\left(|S - Y|_{h,H} - | \widetilde S|_{v,V}\right) + r^2\left(|S|_v -|S|_{V} \right)\,.
\end{align*}
We furthermore have 
$$r|S|_{v,V} - |Y + S|_{h,H} - 1 =|\widetilde S|_{v,V} + 1\,,$$
hence we have
$$\alpha =  2r|S|_{h,H} - 2r^2|S|_V + 2r = -2r(r|S|_V - |S|_{h,H} - 1)\,.$$
Since the shadow of $\nu_j$ extends to $\eta_i$, we have $r|S|_V - |S|_{h,H} - 1 = 0$.  We can therefore conclude that $\alpha = 0$, as desired.
\end{proof}

\begin{figure}
\[
\begin{tikzpicture}[scale=.65]
\def\v{0.15cm}
\draw[step=1,color=gray] (0,0) grid (21,8);
\draw[line width=2,color=green] (0,0)--(3,0)--(3,1)--(6,1)--(6,2)--(8,2)--(8,3)--(11,3)--(11,4)--(14,4)--(14,5)--(16,5)--(16,6);
\draw[line width=2,color=green] (11,4)--(14,4);
\draw[line width=2,color=green] (14,5)--(16,5);
\draw[line width=2,color=orange] (18,6) node[circle,fill=black,minimum size=\v,inner sep=0pt]{}--(19,6)--(19,7)--(21,7)--(21,8);
\draw[color = white, line width=3pt] (16,5)--(16,6);
\draw[color = white, line width=3pt] (14,4)--(14,5);
\draw[color = white, line width=3pt] (11,3)--(11,4);
\draw[color = white, line width=3pt] (16,5)--(16,6);
\draw[color = white, line width=3pt] (3,0)--(3,1);
\draw[color = white, line width=3pt] (6,1)--(6,2);
\draw[color = white, line width=3pt] (8,2)--(8,3);
\draw[decorate,decoration=zigzag, color = green, line width=2pt] (16,5)--(16,6);
\draw[decorate,decoration=zigzag, color = green, line width=2pt] (14,4)--(14,5);
\draw[decorate,decoration=zigzag, color = green, line width=2pt] (11,3)--(11,4);
\draw[decorate,decoration=zigzag, color = green, line width=2pt] (3,0)--(3,1);
\draw[decorate,decoration=zigzag, color = green, line width=2pt] (6,1)--(6,2);
\draw[decorate,decoration=zigzag, color = green, line width=2pt] (8,2)--(8,3);
\draw[line width=2,color=purple] (17,6)  node[circle,fill=black,minimum size=\v,inner sep=0pt]{}--(18,6);
\draw[line width=2,color=black, densely dashed] (16,6) node[circle,fill=black,minimum size=\v,inner sep=0pt]{}--(17,6);
\draw (15.6,6.7) node[anchor=north,color=black]  {\small$p_1$};
\draw (16.6,6.7) node[anchor=north,color=black]  {\small$p_2$};
\draw (17.6,6.7) node[anchor=north,color=black]  {\small$p_3$};
\draw (16.5,6) node[anchor=north,color=black]  {\small$\eta_{18}$};
\draw (-1,4.5) node[anchor=north,color=black]  {$\calC_n$};
\end{tikzpicture}
\]
\[
\begin{tikzpicture}[scale=.65]
\def\v{0.15cm}
\draw[step=1,color=gray] (0,0) grid (8,3);
\draw[color = white, line width=2pt] (6,1)--(6,2);
\draw[line width=2,color=green] (0,0)--(2,0)--(3,0)--(3,1)--(6,1);
\draw[line width=2,color=black, densely dashed] (6,1)--(6,2);
\draw[line width=2,color=orange] (6,2)--(8,2)--(8,3);
\draw (-1,2) node[anchor=north,color=black]  {$\calC_{n-1}$};
\draw (6.5,1.8) node[anchor=north,color=black]  {\small$\nu_2$};
\end{tikzpicture}
\]
\caption{An illustration of the construction in \autoref{lem: add shadowed left-aligned}.}
\end{figure}
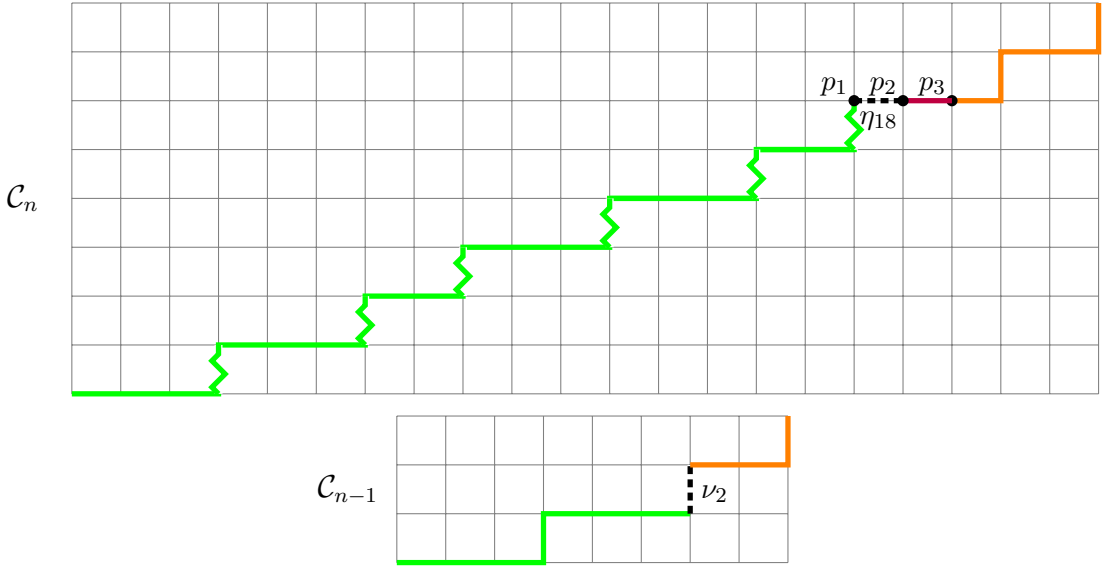

\begin{lem}\label{lem: add shadowed left-aligned}
    Fix a positive compatible pair $(S_1,S_2)$ on $\calC_n$. Suppose that the horizontal edge $\eta_i \notin S_1$ is cascade-paired with $\nu_j \in S_2$ and that $\eta_i$ is to the right of $\nu_j$. Then 
    $$w_q(S_1 \cup \{\eta_i\},S_2) - w_q(S_1,S_2) = w_q(\widetilde S_1, \widetilde S_2 \cup \{\nu_{j'}\}) - w_q(\widetilde S_1,\widetilde S_2)$$
\end{lem}
\begin{proof}
We proceed via analogous methods to the proof of \autoref{lem: add shadowed}.  We start by decomposing $\calC_n$ into several paths.  Let $p_1$ be the left endpoint of $\eta_i$ and let $p_2$ be the right endpoint of $\eta_i$.  Let $p_3$ be the rightmost point along $\calC_n$ at the same height as $\eta_i$. Let $S$ be subpath of $\calC_n$ consisting of all edges below $p_1$, and let $Y$ be the path from $p_1$ to $p_2$. Let $S$ be the path from $p_2$ to $p_3$, and let $Z$ be the subpath of $\calC_n$ consisting of all edges to the right of or above $p_3$.  Note that $\sh(\nu_j;S_2) = S \cup Y \cup \{\eta_i\}$.

We similarly decompose $\calC_{n-1}$ into several paths.  Let $\widetilde S$ be the image of $S$ with $\nu_{j'}$ removed, and let $\widetilde Z$ be the image of $\{\eta_i\} \cup Y \cup Z$.

Expanding out the pairs of edges that involve $\eta_i$ in $\calC_n$, we have
\begin{align*}
w_q(S_1 \cup \{\eta_i\},S_2) - w_q(S_1,S_2) &= w_q(SYHZ - SYhZ) \\
&= r|Y+Z - S|_{h,H} + r^2|S - Y- Z|_V\\
&= r|Y+Z - S|_{h,H} + r^2|S - Z|_V 
\end{align*}
and 
$$w_q(\widetilde S_1, \widetilde S_2\cup \{\eta_j\}) - w_q(\widetilde S_1,\widetilde S_2) = r|\widetilde S - \widetilde Z|_{v,V} + r^2|\widetilde Z - \widetilde S|_H\,.$$
Thus, we want to show that the quantity 
\begin{align*}
\alpha &\colonequals r\left(|Y+Z - S|_{h,H} + |\widetilde Z - \widetilde S|_{v,V}\right) + r^2\left(|S- Z|_V+ |\widetilde S - \widetilde Z|_H\right)
\end{align*}
vanishes.  

By the definition of the map $\rho$, we have 
$$r|S|_{v,V} - |S|_{h,H} = |\widetilde S|_{v,V} + 1$$
and 
$$r|Z|_{v,V} - |Y+Z|_{h,H} - 1 = |\widetilde Z|_{v,V}\,.$$
Substituting these into our expression for $\alpha$, we have
\begin{align*}
\alpha &= r\left(|Y+Z|_{h,H} + |\widetilde Z|_{v,V}\right) + r^2\left(-|Z|_V+ |\widetilde S - \widetilde Z|_H - |S|_v\right) + r\\
&= r^2\left(|Z-S|_v + |\widetilde S - \widetilde Z|_H\right)\,.
\end{align*}

We furthermore have that $|S|_v = |\widetilde S|_H$ and $|Z|_v = |\widetilde Z|_H$.  We can thus conclude that $\alpha = 0$, as desired.
\end{proof}

\begin{lem}\label{lem: flip unshadowed}
For any compatible pair $(S_1,S_2)$ with $S_1 \subseteq \cas(S_2)$, we have 
$$w_q(S_1 \cup (\calP_1 \cut \sh(S_2)),S_2) = w_q(S_1,S_2)\,.$$
\end{lem}
\begin{proof}
This follows because, for any $m \in \Z_{\geq 0}$, we have
$$w_q(Vh + mhh + (r-m)Hh) = w_q(VH + mhH + (r-m)HH)$$
and $w_q(hv) = w_q(Hv)$.
\end{proof}

\begin{lem}\label{lem: ttheta quantum preservation}
If $rb \leq \ell$ and the path $\calP(\ell,h)$ corresponds to a cluster monomial, then the restriction of the $\ttheta$ to the set $\CP_{\cas}(\ell,h,a,b)$ preserves the quantum weight of compatible pairs.
\end{lem}
\begin{proof}
We first see that $\ttheta$ preserves quantum weights when $S_1$ is empty, since the quantum weight of $(\emptyset,S_2)$ only depends on the number of $i,j \in [h]$ such that $\nu_i \in S_2$ and $\nu_j \notin S_1$.  An analogous statement holds for $\ttheta(\emptyset,S_2) = (\emptyset,\emptyset)$.

One can then add on horizontal edges to the compatible pair one-by-one.  The fact that the quantum weight is preserved under this operation follows from \autoref{lem: add shadowed} and \autoref{lem: add shadowed left-aligned}.
\end{proof}

We lastly show that the compatible pairs satisfy an analogous recursion to that for the broken lines \autoref{lem: broken line recursion}.

\begin{lem}\label{lem: cp recursion}
If $rb \leq c_{n+1}$, We have that
$$|\CP(c_{n+1},c_n,a,b)|_q = \sum_{t=0}^{a} \binom{c_{n+1} - rb}{t}_{q^{2r}}|\CP(c_n,c_{n-1},r(a-t)-b,a - t)|_q\,.$$
\end{lem}
\begin{proof}
We first note that
$$|\CP(c_{n+1},c_n,,a,b)|_q = \bigcup_{t=0}^{a} |\{(S_1,S_2) \in \CP(c_{n+1},c_n,a,b): |\cas(S_2) \cap S_1| = t\}|\,.$$
Hence, by \autoref{prop: gen hor swap} and \autoref{lem: ttheta quantum preservation}, we have
$$|\CP(c_{n+1},c_n,,a,b)|_q = \sum_{t=0}^a\binom{c_{n+1}-rb}{t}_{q^{2r}}|\CP_{\cas}(c_{n+1},c_n,a-t,b)|_q\,.$$
Then, applying $\theta$ to $\CP_{\cas}(c_{n+1},c_n,a-t,b)$ and using \autoref{lem: flip unshadowed}, we have
\begin{align*}|\CP_{\cas}(c_{n+1},c_n,a-t,b)|_q &= |\{(S_1,S_2) \in \CP(c_{n},c_{n-1},c_n-b,a-t) : \calP_1 \cut \cas(S_2) \subseteq S_1\}|_q\\
&= |\CP_{\cas}(c_{n},c_{n-1},r(a-t)-b,a-t)|_q\,.
\end{align*}
The desired equality follows from substituting the above expression in the summation.
\end{proof}

\subsection{From Cluster Variables to Cluster Monomials}
The quantum analogs of cluster monomials for $\calA_q(r,r)$ the \emph{bar-invariant quantum cluster monomials}, which are elements in $\calT$ of the form $q^{\alpha\beta}X_n^\alpha X_{n+1}^\beta$.

While we have restricted the proofs in this section to the cluster variable case for simplicity, the same methods can be readily adapted to the case of cluster monomials.  In this subsection, we outline the main steps of this adaptation to the cluster monomial case.

Let $c_n[\alpha,\beta] = \alpha c_n + \beta c_{n-1}$.  Note that the sequence $c_n[\alpha,\beta]$ still satisfies the recursion $c_{n+1}[\alpha,\beta] = rc_n[\alpha,\beta]-c_{n-1}[\alpha,\beta]$, and that $c_n = c_n[1,0]$.

\begin{rem}\label{rem: broken line cluster monomial recursion}
The statement and proofs of \autoref{cor: cn identities}, \autoref{prop: partition formula negative am}, and \autoref{lem: broken line recursion} hold true if $c_n$ is replaced with $c_n[\alpha,\beta]$ whenever $n \in \{h-\ell + 1, h-1,h,h+1, h + \ell\}$.  Thus, we have 
$$|\BL_-(c_{n+1}[\alpha,\beta],c_n[\alpha,\beta],a,b)|_q = \sum_{t=0}^a \binom{c_n[\alpha,\beta]-b}{a-t}_{q^{2r}}|\BL_-(c_{n}[\alpha,\beta],c_{n-1}[\alpha,\beta],rt-b,t)|_q\,.$$
\end{rem}

In order to study the compatible pairs associated to cluster monomials, we first address how to construct cascades.

\begin{rem}
The path $\calP(\alpha c_n + \beta c_{n+1}, \alpha c_{n-1} + \beta c_n)$ can be constructed by appending copies of the paths $\calP(c_n,c_{n-1})$ or $\calP(c_{n+1},c_n)$.  Proceeding from left to right, the $i^{\Th}$ subpath appended is $\calP(c_{n+1},c_n)$ whenever the $i^{\Th}$ edge of $\calP(\alpha,\beta)$ is horizontal and is $\calP(c_n,c_{n-1})$ otherwise. 
\end{rem}

\begin{defn}
Let $\calP = \calP(\alpha c_n + \beta c_{n+1}, \alpha c_{n-1} + \beta c_n)$ be a maximal Dyck path corresponding to a cluster monomial in $\calA(r,r)$.  Decompose $\calP$ into $\alpha + \beta$ subpaths $\{\calP_i\}_{1 \leq i \leq \alpha + \beta}$ of the form $\calP(c_n,c_{n-1})$ or $\calP(c_{n+1},c_n)$.  Then, for $S_2 \subset \calP_2$, the cascade of $S_2$ is computed locally on each subpath $\calP_i$ according to \autoref{defn: cascade}.
\end{defn}

Having constructed cascades, the definition of the map $\ttheta$ can then be extended directly to the cluster monomial case.  The proof that $\ttheta$ preserves quantum weights follows identically from that in \autoref{subsec: preserving quantum weight} except that now cascades are considered locally on each subpath $\calP_i$.  That is, one can readily check that the quantum weight is preserved for compatible pairs with no horizontal edge, and then one can iteratively add in horizontal edges using \autoref{lem: add shadowed} and \autoref{lem: add shadowed left-aligned}.  From this, we can conclude

$$|\CP(c_{n+1}[\alpha,\beta],c_n[\alpha,\beta],a,b)|_q = \sum_{t=0}^a \binom{c_n[\alpha,\beta]-a}{a-t}_{q^{2r}}|\CP(c_{n}[\alpha,\beta],c_{n-1}[\alpha,\beta],rt-b,t)|_q\,.$$
Combining this with \autoref{rem: broken line cluster monomial recursion}, we can conclude that, in the cluster monomial setting, $\ttheta$ is a weighted bijection that preserves quantum weights.  Thus, we can conclude that $\phi$ is a $q$-weighted bijection in the cluster monomial case.

\begin{thm}\label{thm: phi q weighted bijection}
    If $rb \leq \ell$ and the path $\calP(\ell,h)$ corresponds to a cluster monomial, then the map $\phi: \CP(\ell,h,a,b) \to \BL_-(\ell,h,a,b)$ is a $q$-weighted bijection.
\end{thm}
\begin{proof}
By the above discussion, it is enough to consider the cluster variable case.  The map $\phi$ is defined by taking successive applications of the map $\ttheta$ and extracting the associated quantum binomial coefficient, which also corresponds to a bending of a broken line at one wall of the scattering diagram.  A recursion for the quantum weights of the compatible pairs in terms of $\ttheta$ is given in \autoref{lem: cp recursion} that matches the recursion given for broken lines of negative angular momentum \autoref{lem: broken line recursion}.  Thus, we can conclude that $\phi$ is a bijection that preserves quantum weights.
\end{proof}

By \autoref{lem: almost all}, \autoref{thm: phi q weighted bijection} applies to almost all compatible pairs on paths corresponding to cluster monomials.

\section{A Bijection for the Kronecker Theta Basis}\label{sec: bijection for kronecker}
In this section, we now study a map between all compatible pairs and broken lines corresponding to theta basis elements of the quantum Kronecker cluster algebra $\calA_q(2,2)$. We then show that the theta basis for the quantum Kronecker cluster algebra coincides with many other bases, including the quantum triangular, greedy, and bracelets bases.

\subsection{Positive angular momentum}

   In order to handle the positive angular momentum case, we first extend our map $\ttheta$. The fact that this map indeed preserves compatibility is shown in \autoref{prop: ttheta kron compatibility}. We will be primarily interested in the cascade structure of $\calP(n-3,n-2)$, which is obtained from $\calC_n$ by a reflection.  Since we are working in the reflected Dyck path, we take the convention that $|S_1| = b$ and $|S_2| = a$ in this subsection.

\begin{defn}\label{defn: tilde theta kron}
Suppose $2a \leq h-1$ and $r = 2$.  We define a map $\ttheta:\CP_{\cas}(h-1,h,b,a) \to \CP(h,h+1,h-a,b)$ taking a compatible pair $(S_1,S_2)$ on $\calP = \calP(h-1,h)$ to a compatible pair $(\widetilde S_1, \widetilde S_2)$ on $\calP' = \calP(h,h+1)$ by
\begin{itemize}
\item $\widetilde S_1 = \{\eta_i \in \calP': \nu_i \in \calP \cut S_2 \}$, 
\item $\widetilde S_2 \cut \{\nu_h,\nu_{h+1}\} = \{\nu_j \in \calP' : \nu_{\pro,j} \text{ is overflowing}\}$, and
\item $\nu_h \in \calP'$  (\resp $\nu_{h+1} \in \calP'$) is included in $\widetilde S_2$ if the horizontal edge that is $1$-cascade-paired (\resp $2$-cascade-paired) with $\nu_h \in \calP$ is in $S_1$. 
\end{itemize}    
\end{defn}

\begin{figure}[h]
\[
\begin{tikzpicture}[scale=.3];
\draw[line width=1,color=black] (0,0)--(1,0)--(1,1)--(2,1)--(2,2)--(3,2)--(3,3)--(4,3)--(4,4)--(5,4)--(5,5)--(6,5)--(6,6)--(7,6)--(7,7)--(8,7)--(8,8)--(9,8)--(9,9)--(10,9)--(10,10)--(11,10)--(11,11)--(12,11)--(12,13);
\draw[line width=2pt, color=red] (1,0)--++(0,1);
\draw[line width=2pt, color=red] (3,2)--++(0,1);
\draw[line width=2pt, color=red] (4,3)--++(0,1);
\draw[line width=2pt, color=red] (11,10)--++(0,1);
\draw[line width=2pt, color=red] (12,11)--++(0,1);
\draw[line width=2pt, color=red] (12,12)--++(0,1);
\draw[line width=2pt, color=red] (5,5)--++(1,0);
\draw[line width=2pt, color=red] (4,4)--++(1,0);
\draw[line width=2pt, color=red] (7,7)--++(1,0);
\draw[line width=2pt, color=red] (8,8)--++(1,0);

\draw (15.5,7.5) node[anchor=west]  {\large$\ttheta$};
\draw[-Stealth] (13,6)--(20,6);
\draw[line width=1,color=black,xshift=600] (0,0)--(1,0)--(1,1)--(2,1)--(2,2)--(3,2)--(3,3)--(4,3)--(4,4)--(5,4)--(5,5)--(6,5)--(6,6)--(7,6)--(7,7)--(8,7)--(8,8)--(9,8)--(9,9)--(10,9)--(10,10)--(11,10)--(11,11)--(12,11)--(12,12)--(13,12)--(13,14);
\draw[line width=2pt, color=red,xshift=600] (13,12)--++(0,1);
\draw[line width=2pt, color=red,xshift=600] (12,11)--++(0,1);
\draw[line width=2pt, color=red,xshift=600] (1,0)--++(0,1);
\draw[line width=2pt, color=red,xshift=600] (4,3)--++(0,1);
\draw[line width=2pt, color=red,xshift=600] (1,1)--++(1,0);
\draw[line width=2pt, color=red,xshift=600] (4,4)--++(1,0);
\draw[line width=2pt, color=red,xshift=600] (5,5)--++(1,0);
\draw[line width=2pt, color=red,xshift=600] (6,6)--++(1,0);
\draw[line width=2pt, color=red,xshift=600] (7,7)--++(1,0);
\draw[line width=2pt, color=red,xshift=600] (8,8)--++(1,0);
\draw[line width=2pt, color=red,xshift=600] (9,9)--++(1,0);
\end{tikzpicture}
\]
\label{fig: ttheta kron example}
\caption{An illustration of the map $\ttheta$ applied to the compatible pair $(S_1,S_2)$ on $\calP(12,13)$, where $S_1 = \{\eta_5,\eta_6,\eta_8,\eta_9\}$ and $S_2 = \{\nu_1,\nu_3,\nu_4,\nu_{11},\nu_{12},\nu_{13}\}$.}

\end{figure}
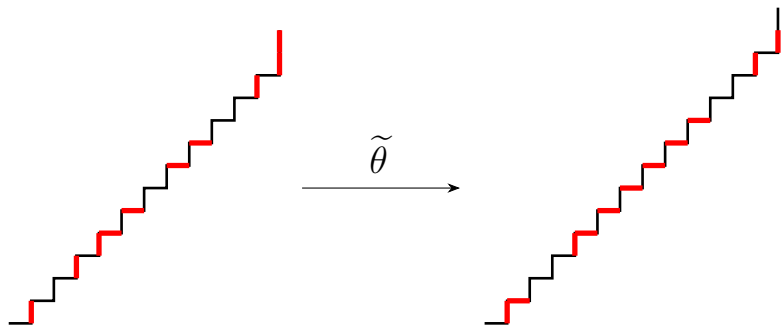

\begin{rem}
In a sense, \autoref{defn: tilde theta kron} is a natural extension of \autoref{defn: tilde theta}.  This is because, by definition, a vertical edge is overflowing if and only if the horizontal edge that is $r$-cascade-paired with it is in $S_1$.  The only difference is that in this case, we need to also consider whether the horizontal edge that is $1$-cascade-paired with each vertical edge is in $S_1$.  This is only possible for the vertical edge $\nu_h \in \calP$, hence why the last condition is added to handle this case.
\end{rem}

\begin{prop}\label{prop: ttheta kron compatibility}
If $2a \leq h-1$, then the map $\ttheta:\CP_{\cas}(h-1,h,b,a) \to \CP(h,h+1,h-a,b)$ preserves compatibility.
\end{prop}
\begin{proof}
It is a straightforward consequence of the definition of compatibility that a pair $(S_1,S_2)$ on $\calP(h-1,h)$ is compatible if and only if, whenever $\nu_j$ is in $S_2$, there is no vertical edge at height $j-1$ or $j-2$ (where the height of $\eta_1$ is $0$).  

Suppose $(S_1,S_2)$ is a compatible pair on $\calP(h-1,h)$, and let $\ttheta(S_1,S_2) = (\widetilde S_1, \widetilde S_2)$.  Suppose $\nu_j \in \widetilde S_2$ for some $j < h$.  Then $\nu_j \in S_2$ is overshadowing, which implies that $j = 1$ or $\nu_{j-1} \in S_2$.  Hence, $\eta_j$ and $\eta_{j-1}$ are not in $\widetilde S_1$.  Moreover, if $\nu_h \in S_2$, then $\eta_h \notin \widetilde S_1$.  Combining this with the characterization of compatible pairs above, we can conclude that $(\widetilde S_1, \widetilde S_2)$ is compatible.
\end{proof}

Another distinction of the positive angular momentum case is that the size of $S_2$ is only non-decreasing, rather than strictly decreasing as in the negative angular momentum case.  Let $(S_1^{(i)},S_2^{(i)}) = \underbrace{\ttheta \circ \ttheta \circ \cdots \circ \ttheta}_{i}(S_1,S_2)$ for $i \geq 0$.  Let $a_i = |S_2^{(i)}|$.  We then set $a_{\infty} = \lim_{i \to \infty} a_i$.  

\begin{figure}[h]
\[
\begin{tikzpicture}[scale=.3];
\draw[line width=1,color=black] (0,0)--(1,0)--(1,1)--(2,1)--(2,2)--(3,2)--(3,3)--(4,3)--(4,4)--(5,4)--(5,5)--(6,5)--(6,6)--(7,6)--(7,8);
\draw[line width=2pt, color=red] (1,0)--++(0,1);
\draw[line width=2pt, color=red] (7,6)--++(0,2);
\draw[line width=2pt, color=red] (1,1)--++(1,0);
\draw[line width=2pt, color=red] (4,4)--++(1,0);
\draw[line width=2pt, color=red] (3,3)--++(1,0);
\draw (11,5.5) node[anchor=west]  {\large$\ttheta$};
\draw[-Stealth] (9,4)--++(6,0);
\draw[line width=1,color=black,xshift=450] (0,0)--(1,0)--(1,1)--(2,1)--(2,2)--(3,2)--(3,3)--(4,3)--(4,4)--(5,4)--(5,5)--(6,5)--(6,6)--(7,6)--(7,7)--(8,7)--(8,9);
\draw[line width=2pt, color=red,xshift=450] (1,0)--++(0,1);
\draw[line width=2pt, color=red,xshift=450] (8,7)--++(0,2);
\draw[line width=2pt, color=red,xshift=450] (1,1)--++(1,0);
\draw[line width=2pt, color=red,xshift=450] (2,2)--++(1,0);
\draw[line width=2pt, color=red,xshift=450] (3,3)--++(1,0);
\draw[line width=2pt, color=red,xshift=450] (4,4)--++(1,0);
\draw[line width=2pt, color=red,xshift=450] (5,5)--++(1,0);
\draw[xshift=450] (11,5.5) node[anchor=west]  {\large$\ttheta$};
\draw[-Stealth,xshift=450] (9,4)--++(6,0);
\draw[line width=1,color=black,xshift=900] (0,0)--(1,0)--(1,1)--(2,1)--(2,2)--(3,2)--(3,3)--(4,3)--(4,4)--(5,4)--(5,5)--(6,5)--(6,6)--(7,6)--(7,7)--(8,7)--(8,8)--(9,8)--(9,10);
\draw[line width=2pt, color=red,xshift=900] (1,0)--++(0,1);
\draw[line width=2pt, color=red,xshift=900] (9,8)--++(0,2);
\draw[line width=2pt, color=red,xshift=900] (1,1)--++(1,0);
\draw[line width=2pt, color=red,xshift=900] (2,2)--++(1,0);
\draw[line width=2pt, color=red,xshift=900] (3,3)--++(1,0);
\draw[line width=2pt, color=red,xshift=900] (4,4)--++(1,0);
\draw[line width=2pt, color=red,xshift=900] (5,5)--++(1,0);
\draw[line width=2pt, color=red,xshift=900] (6,6)--++(1,0);
\end{tikzpicture}
\]
\label{fig: kron b infinity}
\caption{This figure depicts a compatible pair where $|S_2| = |S_2^{(i)}|$ for all $i \geq 0$.  In this example, we have $a_{\infty} = 3$.}
\end{figure}
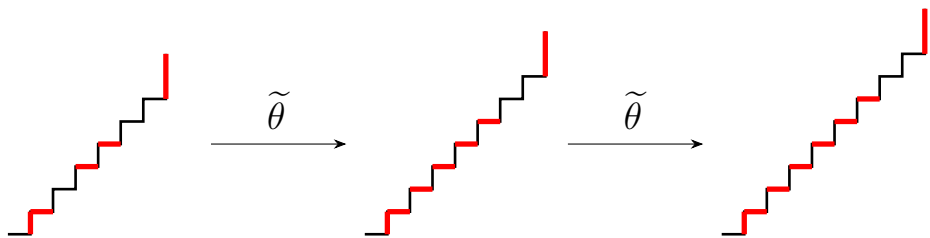

We can characterize the compatible pairs for which applying $\ttheta$ does not decrease $|S_2|$ as follows.

\begin{prop}\label{prop: ttheta stable}
Suppose $(S_1,S_2)$ is a compatible pair on $\calP(h-1,h)$, and let $\ttheta(S_1,S_2) = (\widetilde S_1, \widetilde S_2)$  If $|S_2| = |\widetilde S_2|$, then $(S_1,S_2)$ is of the form
\begin{align*}
    S_2 &= \{\nu_k \in \calP : k \in [1,i]\cup [h-j+1,h]\}\,.\\
S_1 &\supset \{\eta_k : k \in [i+1,2i] \cup [h-2j+1, h-j]\}\,.
\end{align*}
Moreover, we have that, if $\ttheta(\widetilde S_1, \widetilde S_2) = (S_1',S_2')$, then $|S_2'| = |\widetilde S_2|$.
\end{prop}
\begin{proof}
Consider a collection $C$ of contiguous vertical edges in $S_2$.  If the collection does not contain $\nu_h$, then $\cas(C)$ contains the horizontal edge preceding each edge in $C$.  These cannot be included $S_1$, so $|\cas(C) \cap S_1| \leq |C|$.  Moreover, if $\nu_j$ is the lowest edge in $C$ for $j \geq 2$, then $\eta_{j-1}$ must not be in $S_1$ in order to maintain compatibility. Hence $|S_1 \cap \cas(C)| \leq |C|-1$. 

The only cases that were excluded in the above consideration are when $\nu_1$ or $\nu_h$ is in $C$. In both cases, the above arguments show that $|\cas(C) \cap S_1| \leq |C|$.  Thus, if $|\cas(C) \cap S_1| \leq |C|$, then $C$ must be a contiguous set of vertical edges at the top or bottom of $\calP$, and $S_1 \cap \cas(C)$ must be maximal while maintaining the compatibility condition.

Lastly, note that if we apply $\ttheta$ to a compatible pair of this form, the resulting compatible pair is also of this form.
\end{proof}

\begin{defn}
The map $\phi$ takes the compatible pair $(S_1,S_2) \in \CP(h-1,h,b,a)$  to the broken line in $\BL_{+}(h,h-1,b,a)$ that crosses the wall $\frakd_1$ with multiplicity $a_{\infty}$ and the wall $\frakd_{(i-1)/i}$ with multiplicity $a_{i-1}-(2a_i-a_{i+1})$.   
\end{defn}

\begin{thm}\label{thm: phi q weighted bijection kron pos}
    If $2a \leq h-1$, the map $\phi$ is a $q$-weighted bijection from $\CP(h-1,h,b,a)$ to $\BL_{+}(h,h-1,b,a)$.
\end{thm}
\begin{proof}
The proof follows similarly to that of \autoref{thm: phi q weighted bijection}, since the map $\phi$ is still defined by taking successive applications of the map $\ttheta$ and extracting the associated quantum binomial coefficient. One can obtain a quantum recursion for the compatible pairs using \autoref{prop: ttheta kron compatibility} and  \autoref{prop: ttheta stable} that matches the recursion given for broken lines of positive angular momentum given in \autoref{lem: bl recursion pos}.  Thus, we can conclude that $\phi$ is a bijection that preserves quantum weights.
\end{proof}

\subsection{\texorpdfstring{The theta basis element with initial exponent $(m,m)$}{The theta basis element with initial exponent (m,m)}}

The only elements of the theta basis in $\calA_q(2,2)$ that are not cluster monomials are the theta functions $\vartheta_{Q,(m,m)}$ (where $Q$ is a generic point in the first quadrant).  In contrast to the cluster monomial case, a quantum weighting for the corresponding compatible pairs on $\calP(m,m)$ has not yet been constructed.  We begin by defining such a quantization.  This quantization is defined recursively, rather than as a sum over pairs of edges as in Rupel's quantization.  This definition is convenient for our recursive approach to the bijection, though we expect that a more natural generalization of Rupel's quantization can be formulated.

The quantum weighting will be defined recursively using the map $\theta$ between compatible pairs constructed by Lee--Li--Zelevinsky (see \autoref{def: theta equivalent def} for the definition).  Note that we use the map $\theta$ rather than $\ttheta$ because the statement in \autoref{lem: lambda inverse compatibility} fails for the Dyck path $\calP(m,m)$.  This recursion involves choosing a term of a quantum binomial coefficient according to the image of a compatible pair under $\theta$.

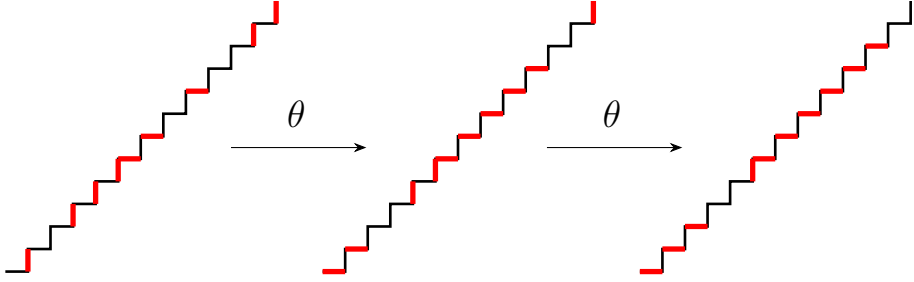
\begin{figure}
\[
\begin{tikzpicture}[scale=.3]
\draw[line width=1,color=black] (0,0)--(1,0)--(1,1)--(2,1)--(2,2)--(3,2)--(3,3)--(4,3)--(4,4)--(5,4)--(5,5)--(6,5)--(6,6)--(7,6)--(7,7)--(8,7)--(8,8)--(9,8)--(9,9)--(10,9)--(10,10)--(11,10)--(11,11)--(12,11)--(12,12);
\draw[line width=2pt, color=red] (1,0)--++(0,1);
\draw[line width=2pt, color=red] (3,2)--++(0,1);
\draw[line width=2pt, color=red] (4,3)--++(0,1);
\draw[line width=2pt, color=red] (5,4)--++(0,1);
\draw[line width=2pt, color=red] (11,10)--++(0,1);
\draw[line width=2pt, color=red] (12,11)--++(0,1);
\draw[line width=2pt, color=red] (5,5)--++(1,0);
\draw[line width=2pt, color=red] (6,6)--++(1,0);
\draw[line width=2pt, color=red] (8,8)--++(1,0);
\draw[line width=1,color=black,xshift=400] (0,0)--(1,0)--(1,1)--(2,1)--(2,2)--(3,2)--(3,3)--(4,3)--(4,4)--(5,4)--(5,5)--(6,5)--(6,6)--(7,6)--(7,7)--(8,7)--(8,8)--(9,8)--(9,9)--(10,9)--(10,10)--(11,10)--(11,11)--(12,11)--(12,12);
\draw (12,7) node[anchor=west]  {\large$\theta$};
\draw[-Stealth] (10,5.5)--(16,5.5);
\draw (26,7) node[anchor=west]  {\large$\theta$};
\draw[-Stealth] (24,5.5)--(30,5.5);
\draw[line width=2pt, color=red,xshift=400] (4,3)--++(0,1);
\draw[line width=2pt, color=red,xshift=400] (5,4)--++(0,1);
\draw[line width=2pt, color=red,xshift=400] (12,11)--++(0,1);
\draw[line width=2pt, color=red,xshift=400] (1,1)--++(1,0);
\draw[line width=2pt, color=red,xshift=400] (0,0)--(1,0);
\draw[line width=2pt, color=red,xshift=400] (5,5)--++(1,0);
\draw[line width=2pt, color=red,xshift=400] (6,6)--++(1,0);
\draw[line width=2pt, color=red,xshift=400] (7,7)--++(1,0);
\draw[line width=2pt, color=red,xshift=400] (8,8)--++(1,0);
\draw[line width=2pt, color=red,xshift=400] (9,9)--++(1,0);
\draw[line width=1,color=black,xshift=800] (0,0)--(1,0)--(1,1)--(2,1)--(2,2)--(3,2)--(3,3)--(4,3)--(4,4)--(5,4)--(5,5)--(6,5)--(6,6)--(7,6)--(7,7)--(8,7)--(8,8)--(9,8)--(9,9)--(10,9)--(10,10)--(11,10)--(11,11)--(12,11)--(12,12);
\draw[line width=2pt, color=red,xshift=800] (5,4)--++(0,1);
\draw[line width=2pt, color=red,xshift=800] (0,0)--++(1,0);
\draw[line width=2pt, color=red,xshift=800] (1,1)--++(1,0);
\draw[line width=2pt, color=red,xshift=800] (2,2)--++(1,0);
\draw[line width=2pt, color=red,xshift=800] (5,5)--++(1,0);
\draw[line width=2pt, color=red,xshift=800] (6,6)--++(1,0);
\draw[line width=2pt, color=red,xshift=800] (7,7)--++(1,0);
\draw[line width=2pt, color=red,xshift=800] (8,8)--++(1,0);
\draw[line width=2pt, color=red,xshift=800] (9,9)--++(1,0);
\draw[line width=2pt, color=red,xshift=800] (10,10)--++(1,0);
\end{tikzpicture}
\]
\caption{The leftmost diagram represents a compatible pair in $\CP_-(12,12,3,6)$.  The middle image is its image under $\theta$, which is a compatible pair in $\CP_-(12,12,7,3)$.  The rightmost image is the compatible pair in $\CP_-(12,12,9,1)$ resulting from another application of $\theta$.  Any subsequent applications of $\theta$ result in the empty compatible pair.}
\end{figure}

The quantum binomial coefficient $\binom{h}{k}_q$ is naturally viewed as a sum over size-$k$ subsets of $[h]$ of monomial terms, i.e,  $\binom{h}{k}_q = q^c\sum_{\substack{J \subseteq [h]//|J| = k}} q^{2\sum_{j \in J} j}$ for some constant $c = c(h,k)$.  Given a size-$k$ subset $J \subseteq [h]$, let $\binom{[h]}{J}_q$ denote the exponent of the monomial corresponding to $J$ in $\binom{h}{k}_q$.

\begin{defn}\label{defn: kron quantum weight}
The quantum weight $w_q(S_1,S_2)$ of a compatible pair $(S_1,S_2)$ on $\calP(m,m)$ can be determined recursively as follows:

As a base case, we set $w_q(\calP_1,\emptyset) = 0$.

Let $U = \calP_1 \cut \sh(S_2) = \{\eta_{i_1},\dots,\eta_{i_h}\}$ be the set of horizontal edges that are not in the shadow of $S_2$, where $i_1 < \cdots < i_h$.  Then  $U \cut S_1 = \{\eta_{i_j} : j \in J\}$ for some $J \subseteq \{1,\dots,h\}$.  We then set 
$$w_q(S_1,S_2) =  w_q(S_1',S_2') + \binom{[h]}{J}_{q^4}\,,$$
where $(S_1',S_2') = \theta(S_1,S_2)$.
\end{defn}

Note that, if $(S_1',S_2') = \theta(S_1,S_2)$, then $|S_1'| \geq \min(m,|S_1| + 1)$ and $|S_2'| \leq \max(0,|S_2| - 1)$.  Thus, any compatible pair is sent to $(\calP_1,\emptyset)$ under finitely many applications of $\theta$, so \autoref{defn: kron quantum weight} is well-defined.  

Using the quantum weighting, we can construct a map from compatible pairs to broken lines that respects quantum weights.  Let $(S_1^{(i)},S_2^{(i)}) = \underbrace{\theta\circ \cdots \circ \theta}_{i}(S_1,S_2)$.

\begin{defn}
Let $\phi: \CP(m,m,a,b) \to \BL_{-}(m,m,a,b)$ be the map taking the compatible pair $(S_1,S_2)$  to the broken line that bending at the wall $\frakd_{(\ell-1)/\ell}$ with multiplicity $h + |S_2^{(\ell + 1)}| -2|S_2^{(\ell)}|-|S_1^{(\ell)}|$.   
\end{defn}

\begin{thm}\label{thm: phi kron theta bijection}
If $2b \leq m$, the map $\phi: \CP(m,m,a,b) \to \BL_{-}(m,m,a,b)$ is a $q$-weighted bijection. 
\end{thm}
\begin{proof}
The size of the set $J$ in applying \autoref{defn: kron quantum weight} to the compatible pair $(S_1^{(\ell)},S_2^{(\ell)})$ is precisely $h + |S_2^{(\ell + 1)}| -2|S_2^{(\ell)}|-|S_1^{(\ell)}|$.  Thus, the binomial coefficient factor that a broken line attains by bending at the wall $\frakd_{(\ell-1)/\ell}$ is the same as the binomial coefficient in \autoref{defn: kron quantum weight} attained by applying $\theta$ to the set of corresponding compatible pairs.
\end{proof}

Note that the positive angular momentum case is symmetric to the negative angular momentum case by reflecting the scattering diagram along the line $y = x$.  As a result of the bijectivity of $\phi$, we obtain an expansion formula in terms of compatible pairs for the quantum theta basis element $\theta_{(m,m),q}$.  

\begin{thm}\label{thm: kron theta basis elt expansion}
    The theta basis element $\vartheta_{Q,(m,m)}$ in $\calA_q(2,2)$ is given by 
    $$\vartheta_{Q,(m,m)} = \sum_{(S_1,S_2)} q^{w_q(S_1,S_2)} X_1^{-m + r|S_2|}X_2^{-m + r|S_1|}\,, $$
    where the sum ranges over all compatible pairs on $\calP(m,m)$.
\end{thm}
\begin{proof}
By definition, the theta function $\vartheta_{Q,(m,m)}$ can be expressed as a sum over broken lines.  Using \autoref{thm: phi kron theta bijection}, the expansion formula in terms of compatible pairs follows directly by applying the bijection $\phi$ to the corresponding broken lines.
\end{proof}

\subsection{Equivalence of bases for the quantum Kronecker cluster algebra}\label{subsec: bases}

Lastly, we show that, for the quantum Kronecker cluster algebra, the quantum theta basis coincides with a number of other bases.  The bases we consider are
\begin{enumerate}
\item\label{it: bracelet} the quantum bracelets basis \cite[Definition 4.9]{Thu14},
\item\label{it: theta} the quantum theta basis \cite[Proposition 3.1]{davison2021strong}*,
\item\label{it: DX} the basis $\mathcal{B}$ consisting of the quantum cluster monomials and $\{z_n\}_{n \geq 1}$ in \cite[Definition 3.4]{DX},
\item\label{it: triangular} the quantum triangular basis \cite[Theorem 1.4]{BZtriangle}, and
\item\label{it: greedy} the quantum greedy basis \cite[Theorem 9]{LLRZ}.
\end{enumerate}
\blfootnote{*This work uses a quantization of the cluster algebra with principal coefficients.  In this case, the mutable variables commute, while the principal coefficients $y_1$ and $y_2$ quasi-commute.  In order to reconcile this with the coefficient-free case, one can reorder $y_1^{\pm 1}$ and $y_2^{\pm 1}$ if necessary and then specialize the principal coefficients to 1.}

The following discussion concerns only the quantum Kronecker cluster algebra, though these bases may coincide in a more general setting.  All bases mentioned contain the cluster monomials, so it is enough to show that the remaining elements coincide.

The fact that (\ref{it: bracelet}) and (\ref{it: theta}) coincide is due to the recent work of Mandel and Qin \cite[Theorem 1.4]{mandel2023bracelets}, who showed that the theta and bracelets bases coincide in a much more general setting.  

To show that (\ref{it: theta}) and (\ref{it: DX}) coincide, it is enough to check that (the specialization of) $\vartheta_{Q,(1,-1,0,0)}$ and $z_1$ are equal.  This is because the $z_n$'s satisfy the Chebyshev recursion of the first kind by definition, while Mandel and Qin \cite[Example 5.8]{mandel2023bracelets} showed that the same recursion holds for the theta functions $\vartheta_{Q,(n,-n,0,0)}$.  By comparing \cite[Lemma 3.3]{canakci2020expansion} and \cite[Lemma 8.4]{mandel2023bracelets}, it is readily seen that $z_1 = \vartheta_{Q,(1,-1,0,0)}|_{y_1 = y_2 = 1}$.

The bases (\ref{it: DX}) and (\ref{it: triangular}) can be shown to coincide by recent work of Li \cite[Theorem 1.2]{Li} which determine the support of triangular basis elements for skew-symmetric rank-$2$ quantum cluster algebras.  It is then straightforward to check that the supports of the elements $z_n$ are equal to the supports of the non-cluster-monomial elements in the triangular basis, so the bases coincide.

We can then show that (\ref{it: DX}) and (\ref{it: greedy}) coincide by directly verifying that the elements $z_n$ satisfy the recursion given in \cite[Theorem 7]{LLRZ}, which characterizes the greedy basis. We do so using the Laurent expansion of $z_n$ given by the work of Ding-Xu. It follows directly from \cite[Proposition 4.4]{DX} that $z_n = s_n - s_{n-2}$, where $s_n$ is the element 
\begin{equation}\label{eq: DX expansion}s_n = \sum_{p+r \leq n} \binom{n-r}{p}_q \binom{n-p}{r}_q q^{-(2p-n)(2r-n)}X_1^{2p-n}X_2^{2r-n}\,.\end{equation}  

This is a special case of the $q$-analogue of the following lemma, which can be easily proved by the standard generating function method.
\begin{lem}
For any nonnegative integers $n,p,r$ with $p \geq r$, we have
$$\sum_{k=0}^p (-1)^k \binom{n-1-r}{p-k}P(k)\binom{n-2r+k-1}{k}=0,
$$where $P(x)$ is any polynomial in $x$ of degree $\le r-1$. 
\end{lem}

\iffalse
In what follows, the binomial coefficient ${A\choose B}$  is defined by \[{A\choose B}:={A(A-1)\cdots(A-B+1)\over B!}\] for any real number $A$ and any nonnegative integer $B$.

\AB{Should these binomials be $q$-binomials?}

\begin{lem}
For any nonnegative integers $n,p,r$ where $p > r$, we have
$$\sum_{k=0}^p (-1)^k {n-r\choose p-k}{n-p+k \choose r}{n-2r+k-1\choose k}=0.
$$
\end{lem}
\begin{proof}
We want to show that whenever $p > r$, we have
$$\sum_{k=0}^p (-1)^k {n-r\choose p-k}{n-p+k \choose r}{n-2r+k-1\choose k}=0,
$$
which is equivalent to
$$\sum_{k=0}^p {n-r\choose p-k}{n-p+k \choose r}{2r-n\choose k}=0.
$$

Observe that the coefficient of $x^{p-m}$ in 
$$\frac{d^m}{dx^m}[(1+x)^{n-r}](1+x)^{2r-n}=(n-r)(n-r-1)\cdots(n-r-m+1)(1+x)^{r-m}$$
is equal to $0$ for any $m\le r$.
So if $P(y)$ is any polynomial in $y$ of degree $\le r$, then
$$\sum_{k=0}^p {n-r\choose p-k}P(k){2r-n\choose k}=0.
$$
The desired identity follows.
\end{proof}
\fi

\begin{rem}
    We note that the basis consisting of the cluster monomials and $\{s_n\}_{n \geq 1}$ that appears in the work of {\c{C}}anak{\c{c}}{\i} and Lampe \cite[Definition 4.15]{canakci2020expansion}* is not identical to these other bases.  It instead coincides (after specializing and possibly reordering the principal coefficients) with the basis $\mathcal{S}$ from the work of Ding-Xu \cite[Definition 3.4]{DX}.  The elements $s_n$ are related to the elements $z_n$ in the basis $\mathcal{B}$, and their Laurent expansion is given in \autoref{eq: DX expansion}.   The work of {\c{C}}anak{\c{c}}{\i} and Lampe was later generalized to quantum cluster algebras from unpunctured orbifolds by Min Huang \cite{Huang}.
\end{rem}

\addtocontents{toc}{\protect\setcounter{tocdepth}{0}}
\section*{Acknowledgements}
We are thankful to {\.I}lke {\c{C}}anak{\c{c}}{\i}, Philipp Lampe, Li Li, Travis Mandel, Fan Qin, and Dylan Rupel for helping us to resolve the connections between various bases for the quantum Kronecker cluster algebra.  We additionally thank Mark Gross, Gregg Musiker, Lang Mou, Salvatore Stella, and Harold Williams for their helpful comments. The first author was supported by the NSF GRFP and the Jack Kent Cooke Foundation. The second author was supported by the University of Alabama, Korea Institute for Advanced Study, and the NSF grants DMS 2042786 and DMS 2302620.

%\nocite{*}
\bibliographystyle{amsplain}
\bibliography{bibliography.bib} 
\addtocontents{toc}{\protect\setcounter{tocdepth}{1}}
\end{document}